\documentclass[smallextended,numbook,runningheads]{svjour3}     
\usepackage{graphicx}
\usepackage{amsmath}
\usepackage{epstopdf} 
\usepackage{mathptmx}      
\journalname{Numer. Math.}
\usepackage{amsmath,ulem}
\usepackage{stmaryrd}
\usepackage{amssymb}
\usepackage{empheq}
\usepackage{cases}
\usepackage{cite}
\usepackage{caption,graphicx}

\newtheorem{prop}{Proposition}[section]

\usepackage{amsmath}
\usepackage{graphicx}
\usepackage{graphicx,epstopdf}
\usepackage{wrapfig}
\usepackage[caption=false]{subfig}
\usepackage{amssymb}
\usepackage{cases}

\usepackage{enumerate}
\usepackage{stmaryrd}
\usepackage{mathrsfs}
\usepackage{multirow}

\usepackage{indentfirst}

\usepackage{array} 

\usepackage[colorlinks,linkcolor=blue]{hyperref}
\hypersetup{backref=true,citebordercolor={1 1 1},citecolor=blue}
\hypersetup{linkbordercolor={1 1 1},linkcolor=blue,linktocpage=true}

\input psfig.sty

\newcommand{\be}{\begin{equation}}
\newcommand{\ee}{\end{equation}}
\newcommand{\ba}{\begin{array}}
\newcommand{\ea}{\end{array}}
\newcommand{\bea}{\begin{eqnarray}}
\newcommand{\eea}{\end{eqnarray}}
\newcommand{\beas}{\begin{eqnarray*}}
\newcommand{\eeas}{\end{eqnarray*}}

\smartqed  
\numberwithin{equation}{section}

\usepackage{newtxtext,newtxmath}

\begin{document}


\title{A structure-preserving parametric finite element method for geometric flows with anisotropic surface energy}

\titlerunning{SP-PFEM for geometric flows with anisotropic surface energy}        

\author{Weizhu Bao \and
        Yifei Li
}

\authorrunning{W.~Bao, and Y.~Li} 

\institute{W.~Bao  \at
              Department of Mathematics, National
              University of Singapore, Singapore 119076\\
              Fax: +65-6779-5452, Tel.: +65-6516-2765,\\
              URL: https://blog.nus.edu.sg/matbwz/\\
              \email{matbaowz@nus.edu.sg}\\[1em]
           Y.~Li \at
              Department of Mathematics, National
              University of Singapore, Singapore 119076\\
              \email{e0444158@u.nus.edu}
              }

\date{Received: date / Accepted: date}

\maketitle


\begin{abstract}
We propose and analyze structure-preserving parametric finite element methods (SP-PFEM) for evolution of a closed curve under different geometric flows
with arbitrary anisotropic surface energy $\gamma(\boldsymbol{n})$ for $\boldsymbol{n}\in \mathbb{S}^1$ representing the outward unit normal vector. By introducing a novel surface energy matrix $\boldsymbol{G}_k(\boldsymbol{n})$ depending
on $\gamma(\boldsymbol{n})$ and the Cahn-Hoffman $\boldsymbol{\xi}$-vector  as well as a nonnegative stabilizing function $k(\boldsymbol{n}):\ \mathbb{S}^1\to \mathbb{R}$, which is a sum of a symmetric positive definite matrix and an anti-symmetric matrix,  we obtain a new geometric partial differential equation and its corresponding variational formulation for the evolution of a closed curve under anisotropic surface diffusion. Based on the new weak formulation, we
propose a parametric finite element method for the anisotropic surface diffusion
and show that it is area conservation and energy dissipation under a very mild condition on $\gamma(\boldsymbol{n})$. The SP-PFEM is then extended
to simulate evolution of a close curve under other anisotropic geometric flows including anisotropic curvature flow and  area-conserved anisotropic curvature flow. Extensive numerical results are reported to demonstrate
the efficiency and unconditional energy stability as well as good mesh quality property of the proposed SP-PFEM for simulating anisotropic geometric flows.


\end{abstract}


\section{Introduction}
Anisotropic surface energy along surface/interface is ubiquitous in solids and materials science due to the lattice orientational difference \cite{gurtin2002interface,Cahn94}.
It thus generates anisotropic evolution process of interface/surface (or geometric flows with anisotropic surface energy) in material sciences
\cite{Sutton95,Thompson12,jiang2012}, imaging sciences \cite{niessen1997general,clarenz2000anisotropic,dolcetta2002area}, and computational geometry \cite{wang2015sharp,Cahn94,deckelnick2005computation}. In fact, anisotropic geometric flows have significant and broader applications in materials science, solid-state
physics and computational geometry, such as grain boundary growth \cite{barrett2010finite}, foam bubble/film \cite{shen2004direct},  surface phase formation \cite{Ye10a},
epitaxial growth \cite{Fonseca14,gurtin2002interface}, heterogeneous catalysis \cite{randolph2007controlling},
solid-state dewetting \cite{Thompson12,bao2017parametric,https://doi.org/10.48550/arxiv.2012.11404}, computational graphics \cite{niessen1997general,clarenz2000anisotropic,dolcetta2002area}.

Assume $\Gamma$ be a closed curve in two dimensions (2D) associated with
a given anisotropic surface energy $\gamma(\boldsymbol{n})$,
where $\boldsymbol{n}=(n_1,n_2)^T\in \mathbb{S}^1$ representing the
unit outward normal vector, and define the free energy functional of the
closed curve $\Gamma$ associate with anisotropic surface energy $\gamma(\boldsymbol{n})$ as
\begin{equation} \label{eq: geometric quantities}
    W(\Gamma):=\int_{\Gamma}\gamma(\boldsymbol{n})ds,
\end{equation}
where $s$ denotes the arc-length parameter of $\Gamma$. By applying
the thermodynamic variation, one can obtain the chemical potential $\mu:=\mu(s)$ (or weighted curvature denoted as $\kappa_\gamma:=\kappa_\gamma(s)$) generated from the energy functional $W(\Gamma)$ as \cite{jiang2019sharp,taylor1992ii}
\begin{equation}\label{chemical-pot}
\mu=\kappa_\gamma:=\frac{\delta W(\Gamma)}{\delta \Gamma}=
\lim_{\varepsilon\to0}\frac{W(\Gamma^\varepsilon)-W(\Gamma)}{\varepsilon},
\end{equation}
where $\Gamma^\varepsilon$ is a small perturbation of $\Gamma$.
Different geometric flows associate with the anisotropic surface energy $\gamma(\boldsymbol{n})$ can be easily defined by
providing the normal velocity $V_n$ for evolution of $\Gamma$.
Typical geometric flows widely used in different applications including
anisotropic curvature flow, area-conserved anisotropic curvature flow
and anisotropic surface diffusion with the corresponding normal velocity
$V_n$  given as \cite{cahn1991stability,taylor1992overview,chen1999uniqueness,andrews2001volume,Barrett2020}
\begin{equation}\label{eqn: aniso geo flow}
V_n=\left\{\begin{array}{ll}
   -\mu,  &\qquad \hbox{anisotropic curvature flow},\\
   -\mu+\lambda, &\qquad \hbox{area-conserved anisotropic curvature flow},\\
   \partial_{ss}\mu, &\qquad \hbox{anisotropic surface diffusion},\\
\end{array}\right.
\end{equation}
where 
$\lambda$ is chosen such that the area of the region enclosed by $\Gamma$ is conserved, which is given as \begin{equation}\label{lambdcon}
\lambda=\frac{\int_{\Gamma} \mu\,ds}{\int_{\Gamma}1\,ds},
\quad \Leftrightarrow \quad \int_{\Gamma}V_n\, ds=\int_{\Gamma}(-\mu+\lambda) ds=0.
\end{equation}
We remark here that if $\gamma(\boldsymbol{n})\equiv 1$ for $\boldsymbol{n}\in \mathbb{S}^1$, i.e. isotropic surface energy,
then the chemical potential defined in \eqref{chemical-pot}
collapses to the curvature $\kappa$ of $\Gamma$, and then
the geometric flows defined in \eqref{eqn: aniso geo flow}
collapse to curvature flow (or curve-shortening flow), area-conserved curvature flow and surface diffusion, respectively, we refer to the review paper \cite{Barrett2020}.


Based on different parametrizations of $\Gamma$ and mathematical formulations for the geometric flows \eqref{eqn: aniso geo flow},
several numerical methods
have been proposed for simulating the geometric flows
\eqref{eqn: aniso geo flow} with isotropic and anisotropic surface energy
corresponding to $\gamma(\boldsymbol{n})\equiv 1$ and $\gamma(\boldsymbol{n})\neq {\rm const}$, respectively. These numerical methods include  the level set method and the phase-field method \cite{du2020phase}, the marker particle method \cite{wong2000periodic,du2010tangent}, the finite element method based on graph formulation \cite{deckelnick2005computation,deckelnick2005fully}, the discontinuous Galerkin method \cite{xu2009local}, the crystalline method \cite{carter1995shape,girao1996crystalline}, the evolving surface finite element method (ESFEM) \cite{kovacs2019convergent}, and the parametric finite element method (PFEM) \cite{barrett2007parametric,barrett2008parametric}.
Due to that only the normal velocity is given in the anisotropic geometric flows \eqref{eqn: aniso geo flow}, different artificial tangential velocities
are introduced  in different numerical methods, which cause different accuracies and mesh qualities, and thus some methods need to carry out re-mesh frequently during evolution to avoid the collision of mesh points.
Among those numerical methods, the PFEM proposed by Barrett, Garcke, and N\"urnberg (also called as BGN scheme in the literature) demonstrates
several good properties including efficiency and accuracy, unconditional
energy stability and asymptotic equal-mesh distribution  for isotropic geometric flows. Recently, by introducing a clever approximation of the normal vector, Bao and Zhao proposed a structure-preserving PFEM (SP-PFEM) \cite{bao2021structurepreserving,bao2022volume} for surface diffusion.
Different techniques have been proposed to extend PFEM or SP-PFEM from
isotropic geometric flows to anisotropic geometric flows in the literatures
\cite{barrett2008numerical,barrett2008variational,wang2015sharp,jiang2016solid,jiang2019sharp,li2020energy}.
Very recently, by introducing a symmetric positive definite surface matrix $\boldsymbol{Z}_k(\boldsymbol{n})$ depending
on $\gamma(\boldsymbol{n})$ and the Cahn-Hoffman $\boldsymbol{\xi}$-vector  as well as a stabilizing function $k(\boldsymbol{n}):\ \mathbb{S}^1\to \mathbb{R}^+$ \cite{bao2021symmetrized,bao2022symmetrized},
we successfully and systematically extended the SP-PFEM method from isotropic surface diffusion to anisotropic surface diffusion under a symmetric (and necessary) condition on  $\gamma(\boldsymbol{n})$ as:
$\gamma(\boldsymbol{n})=\gamma(-\boldsymbol{n})$ for $\boldsymbol{n}\in \mathbb{S}^1$.
Unfortunately, there are different anisotropic surface energy $\gamma(\boldsymbol{n})$ which is not symmetric, e.g.
the $3$-fold anisotropic surface energy \cite{wang2015sharp,bao2017parametric},
in different applications. Thus the proposed
SP-PFEM in \cite{bao2021symmetrized,bao2022symmetrized} cannot be
applied to handle anisotropic geometric flows with non-symmetric surface energy
$\gamma(\boldsymbol{n})$.

In fact, it is still open to design a structure-preserving parametric finite element method (SP-PFEM) for anisotropic geometric flows
\eqref{eqn: aniso geo flow} with general anisotropy $\gamma(\boldsymbol{n})$, especially when $\gamma(\boldsymbol{n})\ne \gamma(-\boldsymbol{n})$. The main aim
of this paper is to attack this important and challenging problem.
The main ingredients in the proposed methods are based on:
(i) the introduction of a novel surface energy matrix $\boldsymbol{G}_k(\boldsymbol{n})$ depending
on $\gamma(\boldsymbol{n})$ and the Cahn-Hoffman $\boldsymbol{\xi}$-vector  as well as a stabilizing function $k(\boldsymbol{n}):\ \mathbb{S}^1\to \mathbb{R}^+$, which can be explicitly decoupled into a symmetric positive definite matrix and an anti-symmetric matrix, (ii) a new conservative geometric partial differential equation (PDE), (iii) a new variational formulation, and (iv) a proper approximation of the normal vector. We prove that the proposed
SP-PFEM is structure-preserving, i.e. area conservation and energy dissipation, for anisotropic surface diffusion under a very mild
condition
\begin{equation}
3\gamma(\boldsymbol{n})>\gamma(-\boldsymbol{n}),\qquad
\boldsymbol{n}\in \mathbb{S}^1.
\end{equation}
Finally we extend the proposed SP-PFEM
to simulate evolution of a closed curve under other anisotropic geometric flows including anisotropic curvature flow and  area-conserved anisotropic curvature flow.

The remainder of this paper is organized as follows: In section 2, by introducing the surface energy matrix $\boldsymbol{G}_k(\boldsymbol{n})$, we derive a new conservative PDE and its corresponding variational formulation for anisotropic surface diffusion, and propose the semi-discretization in space and the full-discretization by SP-PFEM. In section 3, we establish the unconditional energy dissipation of the proposed SP-PFEM for anisotropic surface diffusion. In section 4,
we extend the proposed SP-PFEM  to the anisotropic curvature flow and area-conserved
anisotropic curvature flow. Extensive numerical results are reported in section 5 to illustrate the efficiency and accuracy as well as structure-preserving properties of the proposed SP-PFEM.
Finally, some concluding remarks are drawn in section 6.

\section{The structure-preserving PFEM for anisotropic surface diffusion} In this section, by taking the anisotropic surface diffusion in
\eqref{eqn: aniso geo flow}, we introduce a surface energy matrix $\boldsymbol{G}_k(\boldsymbol{n})$, obtain a conservative geometric PDE and
derive its corresponding variational formulation. An SP-PFEM for the variational problem is presented and its structral-preserving property is
stated under a very mild condition on the anistropic surface energy
$\gamma(\boldsymbol{n})$.

\subsection{The geometric PDE}
Let $\Gamma:=\Gamma(t)$ be parameterized by $\boldsymbol{X}:=\boldsymbol{X}(s, t)=(x(s,t),y(s,t))^T\in{\mathbb R}^2$
with $s$ denoting the time-dependent arc-length parametrization (or `Lagrangian coordinate') of $\Gamma$ and $t$ representing the time. Then the anisotropic surface diffusion
of $\Gamma$ can be described by
\begin{equation}
\partial_t \boldsymbol{X}= \partial_{ss} \mu \, \boldsymbol{n},
\end{equation}
where $\boldsymbol{n}=(n_1,n_2)^T\in {\mathbb S}^1$ represents the unit outward normal vector of $\Gamma$. In practice, another popular way
(or `Eulerian coordinate') is to adopt $\rho\in  \mathbb{T}=\mathbb{R}/\mathbb{Z}=[0, 1]$ being the periodic interval
and then parameterize $\Gamma(t)$ on $\mathbb{T}$ by
$\boldsymbol{X}(\rho, t)$,  which is given as
\begin{equation}
    \Gamma(t):=\boldsymbol{X}(\cdot, t)\quad \boldsymbol{X}(\cdot, t): \mathbb{T}\to \mathbb{R}^2, (\rho, t)\mapsto(x(\rho,t), y(\rho,t))^T.
\end{equation}
Then the arc-length parameter $s$ is given as $s(\rho, t)=\int_0^\rho |\partial_\rho \boldsymbol{X}(\rho, t)|\, d\rho$ satisfying $\partial_\rho s=|\partial_\rho \boldsymbol{X}|$. We assume that the parametrization by $\rho$ is regular during the evolution, i.e., there is a constant $C>1$ such that $\frac{1}{C}\leq |\partial_\rho s(\rho, t)|\leq C$. The normal velocity $V_n$ can be given by this parametrization as
\begin{equation}\label{Vn by X}
    V_n=\boldsymbol{n}\cdot \partial_t \boldsymbol{X}.
\end{equation}

In order to get an explicit formula of $\mu$ from $\gamma(\boldsymbol{n})$,
let $\gamma(\boldsymbol{p})$ be a one-homogeneous extension of $\gamma(\boldsymbol{n})$ from ${\mathbb S}^1$ to ${\mathbb R}^2$, e.g.
\begin{equation}\label{gamma p}\gamma(\boldsymbol{p}):=
\begin{cases}
|\boldsymbol{p}|\,\gamma\left(\frac{\boldsymbol{p}}
{|\boldsymbol{p}|}\right),\qquad &\forall \boldsymbol{p}=(p_1,p_2)^T\in \mathbb{R}^2_*:=\mathbb{R}^2\setminus \{\boldsymbol{0}\},\\
0, &\boldsymbol{p}=\boldsymbol{0},
\end{cases}
\end{equation}
where $|\boldsymbol{p}|=\sqrt{p_1^2+p_2^2}$. Based on this homogeneous extension $\gamma(\boldsymbol{p})$, we can talk the regularity of $\gamma(\boldsymbol{n})$ by referring to the regularity of $\gamma(\boldsymbol{p})$, i.e., $\gamma(\boldsymbol{n})\in C^2(\mathbb{S}^1)\iff \gamma(\boldsymbol{p})\in C^2(\mathbb{R}^2\setminus \{\boldsymbol{0}\})$. Then we introduce the Cahn-Hoffman $\boldsymbol{\xi}$-vector \cite{hoffman1972vector,wheeler1999cahn} as
\begin{equation}\label{xi vector}
    \boldsymbol{\xi}:=\boldsymbol{\xi}(\boldsymbol{n}): \mathbb{S}^1\to \mathbb{R}^2, \quad \boldsymbol{n}\mapsto \boldsymbol{\xi}=(\xi_1, \xi_2)^T:=\nabla \gamma(\boldsymbol{p})|_{\boldsymbol{p}=\boldsymbol{n}}=\gamma(\boldsymbol{n})\boldsymbol{n}+(\boldsymbol{\xi}\cdot \boldsymbol{\tau})\boldsymbol{\tau},
\end{equation}
here $\boldsymbol{\tau}=\partial_s\boldsymbol{X}=\boldsymbol{n}^\perp$ is the unit tangential vector of $\Gamma$, and $\perp$ denoting the clockwise rotation by $\frac{\pi}{2}$. Then an explicit formulation of $\mu$ can be expressed as the $\boldsymbol{\xi}$-vector as \cite{jiang2019sharp}
\begin{equation}\label{def of mu}
    \mu:=\mu(\boldsymbol{n})=-\boldsymbol{n}\cdot \partial_s \boldsymbol{\xi}^\perp.
\end{equation}
Thus another equivalent geometric PDE for anisotropic surface diffusion can given as
\begin{subequations}
\label{eqn:governing continuous sd}
\begin{align}
\label{eqn:governing1 continuous sd}
&\boldsymbol{n}\cdot \partial_t \boldsymbol{X}=\partial_{ss} \mu,\\[0.5em]
\label{eqn:governing2 continuous sd}
&\mu=-\boldsymbol{n}\cdot \partial_s \boldsymbol{\xi}^\perp,\qquad
\boldsymbol{\xi}(\boldsymbol{n})=\nabla \gamma(\boldsymbol{p})|_{\boldsymbol{p}=\boldsymbol{n}}.
\end{align}
\end{subequations}

\subsection{The surface energy matrix}
Introduce the surface energy matrix $\boldsymbol{G}_k(\boldsymbol{n})$ as
\begin{equation}\label{def of G_k}
    \boldsymbol{G}_k(\boldsymbol{n}):=\gamma(\boldsymbol{n})I_2
    -\boldsymbol{n}\boldsymbol{\xi}^T+\boldsymbol{\xi}\boldsymbol{n}^T
    +k(\boldsymbol{n})\boldsymbol{n}\boldsymbol{n}^T:=
    \boldsymbol{G}_k^{(s)}(\boldsymbol{n})+\boldsymbol{G}^{(a)}(\boldsymbol{n}), \quad \forall \boldsymbol{n}\in \mathbb{S}^1,
\end{equation}
where $I_2$ is the $2\times 2$ identity matrix, $k(\boldsymbol{n}): \mathbb{S}^1\to \mathbb{R}$ is a nonnegative stabilizing function to be determined later, and $\boldsymbol{G}_k^{(s)}(\boldsymbol{n})$ and $\boldsymbol{G}^{(a)}(\boldsymbol{n})$ are a symmetric positive matrix
and an anti-symmetric matrix, respectively, which are given as
\begin{equation}\label{def of G_kanti}
    \boldsymbol{G}_k^{(s)}(\boldsymbol{n}):=\gamma(\boldsymbol{n})I_2
    +k(\boldsymbol{n})\boldsymbol{n}\boldsymbol{n}^T, \quad
    \boldsymbol{G}^{(a)}(\boldsymbol{n}):=
    -\boldsymbol{n}\boldsymbol{\xi}^T+\boldsymbol{\xi}\boldsymbol{n}^T, \quad \forall \boldsymbol{n}\in \mathbb{S}^1.
\end{equation}
Then we get the relationship between the weighted curvature $\mu$ and the newly constructed $\boldsymbol{G}_k(\boldsymbol{n})$ as

\begin{lemma}\label{lemma: mu and G_k} The weighted curvature $\mu$ defined in \eqref{def of mu} has the following alternative explicit expression
\begin{equation}\label{alt def of mu}
    \mu \boldsymbol{n}=-\partial_s (\boldsymbol{G}_k(\boldsymbol{n})\partial_s \boldsymbol{X}).
\end{equation}
\begin{proof}
From \cite[Lemma 2.1]{bao2021symmetrized}, we know that
\begin{equation}\label{muxi}
    \mu \boldsymbol{n}=-\partial_s \boldsymbol{\xi}^\perp,\quad \boldsymbol{\xi}^\perp=\gamma(\boldsymbol{n})\boldsymbol{\tau}
    -(\boldsymbol{\xi}\cdot \boldsymbol{\tau})\boldsymbol{n}.
\end{equation}
Thus it suffices to show $\boldsymbol{G}_k(\boldsymbol{n})\partial_s \boldsymbol{X}=\boldsymbol{\xi}^\perp$. Since $\partial_s\boldsymbol{X}=\boldsymbol{\tau}$, by using the definition of $\boldsymbol{G}_k(\boldsymbol{n})$ in \eqref{def of G_k}, we can simplify $\boldsymbol{G}_k(\boldsymbol{n})\partial_s \boldsymbol{X}$ as
\begin{align}\label{Gkmu1}
\boldsymbol{G}_k(\boldsymbol{n})\partial_s \boldsymbol{X}&=\boldsymbol{G}_k(\boldsymbol{n})\boldsymbol{\tau}=(\gamma(\boldsymbol{n})I_2-\boldsymbol{n}\boldsymbol{\xi}^T+\boldsymbol{\xi}\boldsymbol{n}^T+k(\boldsymbol{n})\boldsymbol{n}\boldsymbol{n}^T)\boldsymbol{\tau}\nonumber\\
&=\gamma(\boldsymbol{n})\boldsymbol{\tau}-(\boldsymbol{\xi}\cdot \boldsymbol{\tau})\boldsymbol{n}+(\boldsymbol{n}\cdot \boldsymbol{\tau})\boldsymbol{\xi}+k(\boldsymbol{n})(\boldsymbol{n}\cdot \boldsymbol{\tau})\boldsymbol{n}\nonumber\\
&=\gamma(\boldsymbol{n})\boldsymbol{\tau}-(\boldsymbol{\xi}\cdot \boldsymbol{\tau})\boldsymbol{n}=\boldsymbol{\xi}^\perp.
\end{align}
Combining \eqref{Gkmu1} and \eqref{muxi}, we get the desired equality
\eqref{alt def of mu}.\qed
\end{proof}
\end{lemma}

\begin{remark}The surface energy matrix $\boldsymbol{G}_k(\boldsymbol{n})$ given here can not be symmetric unless $\boldsymbol{n}\boldsymbol{\xi}^T=\boldsymbol{\xi}\boldsymbol{n}^T$, which implies $\gamma(\boldsymbol{n})\boldsymbol{n}=\boldsymbol{n}\boldsymbol{\xi}^T\boldsymbol{n}=\boldsymbol{\xi}\boldsymbol{n}^T=\boldsymbol{\xi}$. This can only happen when $\gamma(\boldsymbol{n})$ is isotropic. For anisotropic $\gamma(\boldsymbol{n})$, $\boldsymbol{G}_k(\boldsymbol{n})$ is essentially different from the symmetric surface energy matrix $Z_k(\boldsymbol{n})$ introduced in \cite{bao2021symmetrized,bao2022symmetrized}. Moreover, if we choose $k(\boldsymbol{n})$ to be $0$, then $\boldsymbol{G}_0(\boldsymbol{n})$ will collapse to the surface energy matrix $G(\theta)$ given in \cite{li2020energy}. This fact also means different from the stabilizing function in \cite{bao2021symmetrized}, in this paper we do not need the stabilizing function $k(\boldsymbol{n})$ to be strictly positive.
\end{remark}

By applying \eqref{alt def of mu}, the geometric PDE for anisotropic surface diffusion \eqref{eqn:governing continuous sd} can be rewritten as the following \textit{conservative form}
\begin{subequations}
\label{eqn:conserve continuous sd}
\begin{align}
\label{eqn:conserve1 continuous sd}
&\boldsymbol{n}\cdot \partial_t \boldsymbol{X}=\partial_{ss} \mu,\\[0.5em]
\label{eqn:conserve2 continuous sd}
&\mu\boldsymbol{n}=-\partial_s(\boldsymbol{G}_k(\boldsymbol{n})\partial_s \boldsymbol{X}).
\end{align}
\end{subequations}

\subsection{A variational formulation and its properties}
We then derive the variational formulation for the conservative form \eqref{eqn:conserve continuous sd}. The functional space with respect to $\Gamma(t)$ can be given as
\begin{equation}
    L^2(\mathbb{T}):=\left\{u: \mathbb{T}\to \mathbb{R}\;|\; \int_{\mathbb{T}} |u(\rho)|^2 d\rho<\infty\right\},
\end{equation}
with the following weighted $L^2$-inner product $(\cdot, \cdot)_{\Gamma(t)}$
\begin{equation}
    \Bigl(u, v\Bigr)_{\Gamma(t)}:=\int_{\Gamma(t)} u(s) v(s)ds=\int_{\mathbb{T}} u(\rho) v(\rho) \partial_\rho s(\rho,t)\, d\rho, \quad \forall u,v\in L^2(\mathbb{T}).
\end{equation}
Since we assume the parameterization is regular; the weighted $L^2$-inner product and the $L^2$ space are equivalent to the usual one, which is independent of $t$. The Sobolev space $H^1(\mathbb{T})$ is defined as
\begin{equation}
    H^1(\mathbb{T}):=\left\{u: \mathbb{T}\to \mathbb{R}\;|\; u\in L^2(\mathbb{T}), \partial_s u\in L^2(\mathbb{T})\right\}.
\end{equation}
Moreover, we can extend the above definitions to the functions in $[L^2(\mathbb{T})]^2$ and $[H^1(\mathbb{T})]^2$.

Multiplying a test function $\phi\in H^1(\mathbb{T})$ to
 \eqref{eqn:conserve1 continuous sd}, integrating over $\Gamma(t)$ and taking integration by parts, we obtain
\begin{equation}\label{variation, sd}
    \Bigl(\boldsymbol{n}\cdot \partial_t\boldsymbol{X}, \phi\Bigr)_{\Gamma(t)}=-\Bigl(\partial_s\mu, \partial_s\phi\Bigr)_{\Gamma(t)}.
\end{equation}
Similarly, by multiplying a test function $\boldsymbol{\omega}=(\omega_1, \omega_2)^T\in [H^1(\mathbb{T})]^2$ to \eqref{eqn:conserve2 continuous sd}, we deduce that
\begin{equation}\label{variation, mu}
    \Bigl(\mu \boldsymbol{n}, \boldsymbol{\omega}\Bigr)_{\Gamma(t)}=\Bigl(\boldsymbol{G}_k(\boldsymbol{n})\partial_s \boldsymbol{X}, \partial_s\boldsymbol{\omega}\Bigr)_{\Gamma(t)}.
\end{equation}

Based on the two equations \eqref{variation, sd} and \eqref{variation, mu}, the new variational formulation for the conservative form \eqref{eqn:conserve continuous sd} can be stated as follows: Suppose the initial curve $\Gamma_0:=\boldsymbol{X}(\cdot, 0)=(x(\cdot, 0), y(\cdot, 0))^T\in [H^1(\mathbb{T})]^2$ and the initial weighted curvature $\mu(\cdot, 0):=\mu_0(\cdot)\in H^1(\mathbb{T})$, then for any $t>0$, find the solution $(\boldsymbol{X}(\cdot, t), \mu(\cdot, t))\in [H^1(\mathbb{T})]^2\times H^1(\mathbb{T})$ satisfying
\begin{subequations}
\label{eqn:weak continuous sd}
\begin{align}
\label{eqn:weak1 continuous sd}
&\Bigl(\boldsymbol{n}\cdot\partial_{t}\boldsymbol{X},
\varphi\Bigr)_{\Gamma(t)}+\Bigl(\partial_{s}\mu,
\partial_s\varphi\Bigr)_{\Gamma(t)}=0,\qquad\forall \varphi\in H^1(\mathbb{T}),\\[0.5em]
\label{eqn:weak2 continuous sd}
&\Bigl(\mu\boldsymbol{n},\boldsymbol{\omega}\Bigr)_{\Gamma(t)}-\Bigl( \boldsymbol{G}_k(\boldsymbol{n})\partial_s \boldsymbol{X},\partial_s\boldsymbol{\omega}\Bigr)_{\Gamma(t)}= 0,\quad\forall\boldsymbol{\omega}\in [H^1(\mathbb{T})]^2.
\end{align}
\end{subequations}

Denote the area of the region enclosed by $\Gamma(t)$ as $A(t)$ and the free interfacial energy of $\Gamma(t)$ as $W(t)$, which are given by
\begin{equation}\label{area cont}
A(t):=\int_{\Gamma(t)}y(\rho, t) \partial_\rho x(\rho, t) d\rho,\quad
W(t):=\int_{\Gamma(t)}\gamma(\boldsymbol{n})\,ds, \qquad t\ge0.
 \end{equation}
We can show that the two geometric properties, i.e. area conservation and energy dissipation,  still hold for the weak formulation
\eqref{eqn:weak continuous sd}.
\begin{prop}[area conservation and energy dissiption] Suppose $\Gamma(t)$ is given by the solution $(\boldsymbol{X}(\cdot, t), \mu(\cdot, t))$ of the variational formulation \eqref{eqn:weak continuous sd}, we have \begin{equation}\label{geo properties, cont, sd}
    A(t)\equiv A(0),\qquad W(t)\leq W(t_1)\leq W(0), \qquad \forall t\geq t_1\geq 0.
\end{equation}
\end{prop}
The proof is similar to \cite[Proposition 2.1]{li2020energy} and details are omitted for brevity.

\subsection{A semi-discretization in space}
To obtain the spatial discretization, suppose $N>2$ be an integer, $h=\frac{1}{N}$ is the mesh size, $\mathbb{T}=[0, 1]=\cup_{j=1}^N I_j:=\cup_{j=1}^N [\rho_{j-1}, \rho_j]$ with $\rho_j=jh, j=0, 1, \ldots N$, and we know $\rho_N=\rho_0$ by periodic. The closed curve $\Gamma(t)=\boldsymbol{X}(\cdot, t)$ is approximated by the closed line segments $\Gamma^h(t)=\boldsymbol{X}^h(\cdot, t)=(x^h(\cdot, t), y^h(\cdot, t))^T$ satisfying $\boldsymbol{X}^h(\rho_j, 0)=\boldsymbol{X}(\rho_j, 0)$. And the discretization for test function space $H^1(\mathbb{T})$ with respect to $\Gamma^h(t)$ is given by the following space of piecewise linear finite element functions
\begin{equation}
    \mathbb{K}^h=\mathbb{K}^h(\mathbb{T}):=\left\{u\in C(\mathbb{T}) | \quad u|_{I_j}\in \mathcal{P}^1(I_j), \forall 1\leq j\leq N\right\},
\end{equation}
where $\mathcal{P}^1(I_j)$ is the set of polynomials defined on $I_j$ of degree $\leq 1$. The mass lumped inner product for two function $u, v\in \mathbb{K}^h(\mathbb{T})$ with respect to $\Gamma^h(t)$ is defined as
\begin{equation}\label{mass lumped inner product}
    \left(u, v\right)^h_{\Gamma^h(t)}:=\frac{1}{2}\sum_{j=1}^N |\boldsymbol{h}_j(t)|\,\left((u\cdot v)(\rho_{j-1}^+)+(u\cdot v)(\rho_j^-)\right).
\end{equation}
where $\boldsymbol{h}_j(t)=\boldsymbol{X}^h(\rho_j,t)-\boldsymbol{X}^h(\rho_{j-1}, t)$, and $u(\rho_j^\pm)=\lim\limits_{\rho\to \rho_j^\pm} u(\rho)$. We note this mass lumped inner product is also valid for $[\mathbb{K}^h]^2$ and the piecewise constant vector-valued functions with possible jump discontinuities at $\rho_j$ for $0\leq j\leq N$.

The discretized unit normal vector $\boldsymbol{n}^h(t)$, unit tangential vector $\boldsymbol{\tau}^h(t)$, and the $\boldsymbol{\xi}$-vector $\boldsymbol{\xi}^h(t)$ are such piecewise constant vectors on $\Gamma^h(t)$, which are given as
\begin{equation}
    \boldsymbol{n}_j^h:=\boldsymbol{n}^h|_{I_j}=-\frac{\boldsymbol{h}_j^\perp}{|\boldsymbol{h}_j|}, \quad \boldsymbol{\tau}_j^h:=\boldsymbol{\tau}^h|_{I_j}=(\boldsymbol{n}_j^h)^\perp,  \quad \boldsymbol{\xi}_j^h:=\boldsymbol{\xi}^h|_{I_j}=\nabla \gamma(\boldsymbol{p})|_{\boldsymbol{p}=\boldsymbol{n}_j^h}.
 \end{equation}
 Here we use $\boldsymbol{n}^h$ to represent $\boldsymbol{n}^h(t)$ for short. To make sure $\boldsymbol{n}^h, \boldsymbol{\tau}^h, \boldsymbol{\xi}^h$ are well defined, we need the following assumption on $\boldsymbol{h}_j(t)$
 \begin{equation}\label{assumption on h}
    \min_{1\leq j\leq N}|\boldsymbol{h}_j(t)|>0,\quad \forall t\geq 0.
 \end{equation}

 After giving all the continuous objects their discretized versions, we can state the spatial semi-discretization as follows: Let $\Gamma^h_0:= \boldsymbol{X}^h(\cdot, 0)\in [\mathbb{K}^h]^2, \mu^h(\cdot, 0)\in \mathbb{K}^h$ be the approximations of $\Gamma_0:=\boldsymbol{X}_0(\cdot), \mu_0(\cdot)$, respectively, with $\boldsymbol{X}^h(\rho_j, 0)=\boldsymbol{X}(\rho_j, 0), \\ \mu^h(\rho_j, 0)=\mu_0(\rho_j)$ for $0\leq j\leq N$, find the solution $(\boldsymbol{X}^h(\cdot, t), \mu^h(\cdot, t))\in [\mathbb{K}^h]^2\times \mathbb{K}^h$, such that
\begin{subequations}
\label{eqn:2dsemi sd}
\begin{align}
\label{eqn:2dsemi1 sd}
&\Bigl(\boldsymbol{n}^h\cdot\partial_{t}\boldsymbol{X}^h,
\varphi^h\Bigr)_{\Gamma^h(t)}^h+\Bigl(\partial_{s}\mu^h,
\partial_s\varphi^h\Bigr)_{\Gamma^h(t)}^h=0,\qquad\forall \varphi^h\in \mathbb{K}^h,\\
\label{eqn:2dsemi2 sd}
&\Bigl(\mu^h\boldsymbol{n}^h,\boldsymbol{\omega}^h\Bigr)_{\Gamma^h(t)}^h-\Bigl( \boldsymbol{G}_k(\boldsymbol{n}^h)\partial_s \boldsymbol{X}^h,\partial_s\boldsymbol{\omega}^h\Bigr)_{\Gamma^h(t)}^h=0,
\quad\forall\boldsymbol{\omega}^h\in[\mathbb{K}^h]^2,
\end{align}
\end{subequations}
where
\begin{equation}
\boldsymbol{G}_k(\boldsymbol{n}^h)|_{I_j}=\gamma(\boldsymbol{n}^h_j)I_2-\boldsymbol{n}_j^h (\boldsymbol{\xi}_j^h)^T+\boldsymbol{\xi}_j^h(\boldsymbol{n}^h_j)^T+k(\boldsymbol{n}^h_j)\,\boldsymbol{n}_j^h(\boldsymbol{n}_j^h)^T,
\end{equation}
and the discretized derivative $\partial_s$ for a scalar and vector valued functions $f$ and $\boldsymbol{f}$, respectively, are given as
\begin{equation}
    \partial_s f|_{I_j}:=\frac{f(\rho_j)-f(\rho_{j-1})}{|\boldsymbol{h}_j|}, \quad \partial_s \boldsymbol{f}|_{I_j}:=\frac{\boldsymbol{f}(\rho_j)-\boldsymbol{f}(\rho_{j-1})}{|\boldsymbol{h}_j|},
\end{equation}
and assumption \eqref{assumption on h} ensures $\partial_s f$ and $\partial_s \boldsymbol{f}$ are piecewise constant functions with possible jump discontinuities at $\rho_j$ for $0\leq j\leq N$, thus the mass lumped inner product terms like $\Bigl(\partial_{s}\mu^h, \partial_s\varphi^h\Bigr)_{\Gamma^h(t)}^h$ in \eqref{eqn:2dsemi sd} are well defined.

Denote the enclosed area and the total energy of the closed line segments $\Gamma^h(t)$ as $A^h(t)$ and  $W^h(t)$, respectively, which are given by
\begin{equation}\label{discrete area}
A^h(t)=\frac{1}{2}\sum_{j=1}^N(x_j^h(t)-x_{j-1}^h(t))(y_j^h(t)+y_{j-1}^h(t)),
\quad W^h(t)=\sum_{j=1}^N|\boldsymbol{h}_j(t)|\gamma(\boldsymbol{n}^h_j),
\end{equation}
where $x_j^h(t):=x^h(\rho_j, t), y_j^h(t):=y^h(\rho_j, t), \forall 0\leq j\leq N$. Then following the same statement in the proof of \cite[Proposition 3.1]{li2020energy}, we know that the spatial discretization \eqref{eqn:2dsemi sd} still preserves the two geometric properties.
\begin{prop}[area conservation and energy dissiption] Suppose $\Gamma^h(t)$ is given by the solution $(\boldsymbol{X}^h(\cdot, t), \mu^h(\cdot, t))$ of \eqref{eqn:2dsemi sd}, then we have
\begin{equation}
    A^h(t)\equiv A^h(0),\quad W^h(t)\leq W^h(t_1)\leq W^h(0), \quad \forall t\geq t_1\geq 0.
\end{equation}
\end{prop}

\subsection{A structure-preserving PFEM}
We then consider the full discretization. Let $\tau$ be the uniform time step, and $\Gamma^m=\boldsymbol{X}^m(\cdot)\in[\mathbb{K}^h]^2$ be the approximation of $\Gamma^h(t_m)=\boldsymbol{X}^h(\cdot, t_m), \forall m\geq 0$, where $t_m:=m\tau$. Suppose $\boldsymbol{h}^m_j:=\boldsymbol{X}^m(\rho_j)-\boldsymbol{X}^m(\rho_{j-1})$, we can similarly give the definitions for the mass lumped inner product $\left(\cdot, \cdot\right)^h_{\Gamma^m}$ as well as the unit normal vector $\boldsymbol{n}^m$, the unit tangential vector $\boldsymbol{\tau}^m$, and the $\boldsymbol{\xi}$-vector $\boldsymbol{\xi}^m$ with respect to $\Gamma^m$. By adopting the backward Euler method, the fully-implicit structure-preserving discretization of PFEM for anisotropic surface diffusion
\eqref{eqn:governing continuous sd} can then be given as:

Suppose the initial approximation $\Gamma^0(\cdot)\in [\mathbb{K}^h]^2$ is given by $\boldsymbol{X}^0(\rho_j)=\boldsymbol{X}_0(\rho_j), \forall 0\leq j\leq N$. For any $m=0, 1, 2, \ldots$, find the solution $(\boldsymbol{X}^m(\cdot)=(x^m(\cdot),y^m(\cdot))^T, \mu^m(\cdot))\in [\mathbb{K}^h]^2\times \mathbb{K}^h$, such that
\begin{subequations}
\label{eqn:2dfull sd}
\begin{align}
\label{eqn:2dfull1 sd}
&\Bigl(\boldsymbol{n}^{m+\frac{1}{2}}\cdot\frac{\boldsymbol{X}^{m+1}-\boldsymbol{X}^m}{\tau},
\varphi^h\Bigr)_{\Gamma^m}^h+\Bigl(\partial_{s}\mu^{m+1},
\partial_s\varphi^h\Bigr)_{\Gamma^m}^h=0,\quad\forall \varphi^h\in \mathbb{K}^h,\\
\label{eqn:2dfull2 sd}
&\Bigl(\mu^{m+1}\boldsymbol{n}^{m+\frac{1}{2}},\boldsymbol{\omega}^h\Bigr)_{\Gamma^m}^h-\Bigl( \boldsymbol{G}_k(\boldsymbol{n}^m)\partial_s \boldsymbol{X}^{m+1},\partial_s\boldsymbol{\omega}^h\Bigr)_{\Gamma^m}^h=0,
\quad\forall\boldsymbol{\omega}^h\in[\mathbb{K}^h]^2,
\end{align}
\end{subequations}
where
\begin{equation}
\boldsymbol{G}_k(\boldsymbol{n}^m)|_{I_j}=\gamma(\boldsymbol{n}^m_j)I_2-\boldsymbol{n}_j^m (\boldsymbol{\xi}_j^m)^T+\boldsymbol{\xi}_j^m(\boldsymbol{n}^m_j)^T+k(\boldsymbol{n}^m_j)\,\boldsymbol{n}_j^m(\boldsymbol{n}_j^m)^T,
\end{equation}
and
\begin{equation}
    \boldsymbol{n}^{m+\frac{1}{2}}:=-\frac{1}{2}\frac{1}{|\partial_\rho \boldsymbol{X}^m|}(\partial_\rho \boldsymbol{X}^m+\partial_\rho\boldsymbol{X}^{m+1})^\perp.
\end{equation}
The SP-PFEM \eqref{eqn:2dfull sd} is fully-implicit, and we can apply Newton's iterative method proposed in \cite{bao2021structurepreserving} to solve it numerically. If $\boldsymbol{n}^{m+\frac{1}{2}}$ is replaced by $\boldsymbol{n}^{m}$, then \eqref{eqn:2dfull sd} becomes semi-implicit. However, the clever choice $\boldsymbol{n}^{m+\frac{1}{2}}$ is critical in preserving the area conservation, and the semi-implicit PFEM can only preserve the energy dissipation property.

\subsection{Main results}
Denote the enclosed area and the total energy of the closed line segments $\Gamma^m$ as $A^m$ and  $W^m$, respectively, which are given by
\begin{subequations}\label{fully discrete area and energy}
\begin{align}
\label{fully discrete area and energy, area}
A^m&=\frac{1}{2}\sum_{j=1}^N(x^m(\rho_j)-x^m(\rho_{j-1})(y^m(\rho_j)
+y^m(\rho_{j-1})),\\
\label{fully discrete area and energy, energy}
W^m&=\sum_{j=1}^N|\boldsymbol{h}_j^m|\gamma(\boldsymbol{n}^m_j).
\end{align}
\end{subequations}

Introduce two auxiliary functions $P_{\alpha}(\boldsymbol{n}, \hat{\boldsymbol{n}}), Q(\boldsymbol{n}, \hat{\boldsymbol{n}})$ as
\begin{subequations}
\begin{align}
\label{def of F_k}
&P_{\alpha}(\boldsymbol{n}, \hat{\boldsymbol{n}}):=2\sqrt{(\gamma(\boldsymbol{n})
+\alpha(\hat{\boldsymbol{n}}\cdot \boldsymbol{n}^\perp)^2)\gamma(\boldsymbol{n})},\quad \forall \boldsymbol{n}\in \mathbb{S}^1, \quad \alpha\ge0,\\
\label{def of F}
&Q(\boldsymbol{n}, \hat{\boldsymbol{n}}):=\gamma(\hat{\boldsymbol{n}})+\gamma(\boldsymbol{n})(\boldsymbol{n}\cdot \hat{\boldsymbol{n}})-(\boldsymbol{\xi}\cdot \boldsymbol{n}^\perp)(\hat{\boldsymbol{n}}\cdot \boldsymbol{n}^\perp), \quad \forall \boldsymbol{n}, \hat{\boldsymbol{n}}\in \mathbb{S}^1,
\end{align}
\end{subequations}
we define the minimal stabilizing function $k_0(\boldsymbol{n})$ as
(its existence will be given in next section)
\begin{equation}\label{def of k0}
    k_0(\boldsymbol{n}):=\inf\{\alpha\geq 0: P_{\alpha}(\boldsymbol{n}, \hat{\boldsymbol{n}})-Q(\boldsymbol{n}, \hat{\boldsymbol{n}})\geq 0, \ \forall \hat{\boldsymbol{n}}\in\mathbb{S}^1\}, \quad \forall \boldsymbol{n}\in \mathbb{S}^1.
\end{equation}
For the SP-PFEM \eqref{eqn:2dfull sd}, we have
\begin{theorem}[structure-preserving]\label{thm: main} Suppose $\gamma(\boldsymbol{n})$ satisfies
\begin{equation}\label{es condition on gamma}
    3\gamma(\boldsymbol{n})>\gamma(-\boldsymbol{n}), \quad \forall \boldsymbol{n}\in \mathbb{S}^1, \quad \gamma(\boldsymbol{p})\in C^2\left(\mathbb{R}^2\setminus \{\boldsymbol{0}\}\right),
\end{equation}
and take the stabilizing function $k(\boldsymbol{n})\geq k_0(\boldsymbol{n})$ for $\boldsymbol{n}\in \mathbb{S}^1$, then the SP-PFEM \eqref{eqn:2dfull sd} preserves the two geometric properties in the discrete level, i.e.,
\begin{equation}\label{geo properties, full, sd}
    A^{m+1}\equiv A^0,\qquad W^{m+1}\leq W^m\leq\ldots \leq W^0, \qquad \forall m\geq 0.
    \end{equation}
\end{theorem}
The area conservation part is a direct result of \cite[Theorem 2.1]{bao2021structurepreserving}. And the energy dissipation will be proved in the next section.

\section{Energy dissipation of the SP-PFEM \eqref{eqn:2dfull sd}}
In this section, we first show if $\gamma(\boldsymbol{n})$ satisfies \eqref{es condition on gamma}, the minimal stabilizing function $k_0(\boldsymbol{n})$ defined in \eqref{def of k0} is well-defined, thus we can always choose a nonnegative stabilizing function $k(\boldsymbol{n})\geq k_0(\boldsymbol{n})$ for $\boldsymbol{n}\in \mathbb{S}^1$. After that, we will use $k_0(\boldsymbol{n})$ to give the proof of the unconditional energy stability part of the main theorem \ref{thm: main}.

\subsection{Existence of the minimal stabilizing function $k_0(\boldsymbol{n})$ defined in \eqref{def of k0}}
From the definition of $k_0(\boldsymbol{n})$ in \eqref{def of k0}, we observe that if $(\hat{\boldsymbol{n}}\cdot \boldsymbol{n}^\perp)^2> 0$, then intuitively for sufficiently large $\alpha$, we know the $P_{\alpha}(\boldsymbol{n}, \hat{\boldsymbol{n}})\geq Q(\boldsymbol{n}, \hat{\boldsymbol{n}})$. But this approach will fail when $(\hat{\boldsymbol{n}}\cdot \boldsymbol{n}^\perp)^2=0$, and this can happen if $\hat{\boldsymbol{n}}=\pm \boldsymbol{n}$, which suggests us to treat the two cases $\hat{\boldsymbol{n}}\cdot \boldsymbol{n}\geq0$ and $\hat{\boldsymbol{n}}\cdot \boldsymbol{n}\le 0$ separately. To simplify the notations, we introduce a compact set $\mathbb{S}^1_{\boldsymbol{n}}$ as
\begin{equation}
    \mathbb{S}^1_{\boldsymbol{n}}:=\{\hat{\boldsymbol{n}}\in \mathbb{S}^1| \boldsymbol{n}\cdot \hat{\boldsymbol{n}}\geq 0\}, \quad \boldsymbol{n}\in \mathbb{S}^1.
\end{equation}
Then $\boldsymbol{n}\cdot \hat{\boldsymbol{n}}\geq 0\iff \hat{\boldsymbol{n}}\in \mathbb{S}^1_{\boldsymbol{n}}$, and $\boldsymbol{n}\cdot \hat{\boldsymbol{n}}\leq 0\iff \hat{\boldsymbol{n}}\in \mathbb{S}^1_{-\boldsymbol{n}}$.
\begin{theorem}
The $k_0(\boldsymbol{n})$ defined as \eqref{def of k0} is bounded if the condition \eqref{es condition on gamma} on $\gamma(\boldsymbol{n})$ is satisfied.
\end{theorem}
\begin{proof}
First we consider the case $\hat{\boldsymbol{n}}\in \mathbb{S}^1_{\boldsymbol{n}}$. Since $\gamma(\boldsymbol{n})$ satisfies \eqref{es condition on gamma}, we know $\gamma(\boldsymbol{p})\in C^2(\mathbb{R}^2\setminus \{\boldsymbol{0}\})$. Thus there exists a constant $C_1>0$, such that
\begin{equation}
    \left\|{\bf H}_{\gamma}(\boldsymbol{p})\right\|_2\leq C_1, \quad \forall \frac{1}{2}\leq |\boldsymbol{p}|^2\leq 1,
\end{equation}
here $\bf{H}_{\gamma}$ is the Hessian matrix of $\gamma(\boldsymbol{p})$, and $\left\|\cdot \right\|_2$ denotes the $2$-norm.

By mean value theorem, we know for all $\hat{\boldsymbol{n}}\in \mathbb{S}^1_{\boldsymbol{n}}$, there exists a constant $0<c<1$ such that
\begin{equation}
    \gamma(\hat{\boldsymbol{n}})=\gamma(\boldsymbol{n})+\boldsymbol{\xi}\cdot (\hat{\boldsymbol{n}}-\boldsymbol{n})+
    \frac{1}{2}(\hat{\boldsymbol{n}}-\boldsymbol{n})^T\, {\bf H}_\gamma\left(c\boldsymbol{n}+(1-c)\hat{\boldsymbol{n}}\right)\,
    (\hat{\boldsymbol{n}}-\boldsymbol{n}).
\end{equation}
It is easy to see that $|c\boldsymbol{n}+(1-c)\hat{\boldsymbol{n}}|\leq 1$ and $|c\boldsymbol{n}+(1-c)\hat{\boldsymbol{n}}|^2\geq c^2+(1-c)^2\geq \frac{1}{2}$. Hence we know
\begin{equation}\label{temp, regularity of gamma}
    \gamma(\hat{\boldsymbol{n}})\leq \gamma(\boldsymbol{n})+\boldsymbol{\xi}\cdot (\hat{\boldsymbol{n}}-\boldsymbol{n})+\frac{C_1}{2}
    |\hat{\boldsymbol{n}}-\boldsymbol{n}|^2.
\end{equation}
And notice that $\boldsymbol{\xi}\cdot \boldsymbol{n}=\gamma(\boldsymbol{n})$, we can then get the following estimate of $Q(\boldsymbol{n}, \hat{\boldsymbol{n}})$:
\begin{align}\label{Qnnh1}
&Q(\boldsymbol{n}, \hat{\boldsymbol{n}})-2\gamma(\boldsymbol{n})\nonumber\\
&\leq\left(\gamma(\boldsymbol{n})+\boldsymbol{\xi}\cdot (\hat{\boldsymbol{n}}-\boldsymbol{n})+\frac{C_1}{2}|\hat{\boldsymbol{n}}-\boldsymbol{n}|^2\right)-\Bigl(\boldsymbol{\xi}\cdot \hat{\boldsymbol{n}}-(\boldsymbol{\xi}\cdot \boldsymbol{n})(\hat{\boldsymbol{n}}\cdot \boldsymbol{n})\Bigr)\nonumber\\
&\qquad+\gamma(\boldsymbol{n})(\boldsymbol{n}\cdot \hat{\boldsymbol{n}})-2\gamma(\boldsymbol{n})\nonumber\\
&=2\gamma(\boldsymbol{n})((\boldsymbol{n}\cdot \hat{\boldsymbol{n}})-1)+\frac{C_1}{2}|\hat{\boldsymbol{n}}-\boldsymbol{n}|^2\nonumber\\
&=\left(\frac{C_1}{2}-\gamma(\boldsymbol{n})\right)|\hat{\boldsymbol{n}}-\boldsymbol{n}|^2.
\end{align}
Here we use the fact $\boldsymbol{\xi}\cdot \hat{\boldsymbol{n}}=(\boldsymbol{\xi}\cdot \boldsymbol{n})(\hat{\boldsymbol{n}}\cdot \boldsymbol{n})+(\boldsymbol{\xi}\cdot \boldsymbol{n}^\perp)(\hat{\boldsymbol{n}}\cdot \boldsymbol{n}^\perp)$ and $|\hat{\boldsymbol{n}}-\boldsymbol{n}|^2=2-2\boldsymbol{n}\cdot \hat{\boldsymbol{n}}$.

On the other hand, using the fact $(\hat{\boldsymbol{n}}\cdot \boldsymbol{n}^\perp)^2=1-(\hat{\boldsymbol{n}}\cdot \boldsymbol{n})^2=\frac{|\hat{\boldsymbol{n}}-\boldsymbol{n}|^2}{2}(1+\boldsymbol{n}\cdot \hat{\boldsymbol{n}})$, we know that for $\alpha>\gamma(\boldsymbol{n})$, it holds
\begin{align}\label{Qnnh2}
P_{\alpha}(\boldsymbol{n}, \hat{\boldsymbol{n}})-2\gamma(\boldsymbol{n})&=
\frac{2\alpha(\hat{\boldsymbol{n}}\cdot \boldsymbol{n}^\perp)^2}{\sqrt{1+
\frac{\alpha}{\gamma(\boldsymbol{n})}(\hat{\boldsymbol{n}}\cdot \boldsymbol{n}^\perp)^2}+1}\nonumber\\
&\geq \frac{2\alpha(\hat{\boldsymbol{n}}\cdot \boldsymbol{n}^\perp)^2}{2\sqrt{1+\frac{\alpha}{\gamma(\boldsymbol{n})}}}
=\frac{\alpha(1+\hat{\boldsymbol{n}}\cdot \boldsymbol{n})}{2\sqrt{1+\frac{\alpha}{\gamma(\boldsymbol{n})}}}
|\hat{\boldsymbol{n}}-\boldsymbol{n}|^2\nonumber\\
&\geq \frac{\sqrt{\gamma(\boldsymbol{n})(\alpha-\gamma(\boldsymbol{n}))}}
{2}|\hat{\boldsymbol{n}}-\boldsymbol{n}|^2.
\end{align}
Combining \eqref{Qnnh1} and \eqref{Qnnh2},
we know that for $\alpha\geq \alpha_1:= \frac{(C_1-2\gamma(\boldsymbol{n}))^2+\gamma^2(\boldsymbol{n})}{\gamma(\boldsymbol{n})}<\infty$, it holds $P_{\alpha}(\boldsymbol{n}, \hat{\boldsymbol{n}})\geq Q(\boldsymbol{n}, \hat{\boldsymbol{n}}), \forall \hat{\boldsymbol{n}}\in \mathbb{S}^1_{\boldsymbol{n}}$.

For the case $\hat{\boldsymbol{n}}\in \mathbb{S}^1_{-\boldsymbol{n}}$, by \eqref{es condition on gamma}, when $\hat{\boldsymbol{n}}=-\boldsymbol{n}$, we know that
\begin{align}
Q(\boldsymbol{n}, -\boldsymbol{n})&=\gamma(-\boldsymbol{n})+\gamma(\boldsymbol{n})
(\boldsymbol{n}\cdot (-\boldsymbol{n}))-(\boldsymbol{\xi}\cdot \boldsymbol{n}^\perp)(-\boldsymbol{n}\cdot \boldsymbol{n}^\perp)\nonumber\\
&<3\gamma(\boldsymbol{n})-\gamma(\boldsymbol{n})=2\gamma(\boldsymbol{n}).
\end{align}
Thus for $\alpha=0:=\alpha_{-\boldsymbol{n}}<\infty$, we know $P_0(\boldsymbol{n}, -\boldsymbol{n})=2\gamma(\boldsymbol{n})>Q(\boldsymbol{n}, -\boldsymbol{n})$. By continuity of $P_{\alpha}$ and $Q$, there exits an open set $\mathcal{O}_{-\boldsymbol{n}, 0}\subset \mathbb{S}^1$ such that $-\boldsymbol{n}\in \mathcal{O}_{-\boldsymbol{n}, 0}$, and for all $\hat{\boldsymbol{n}}\in \mathcal{O}_{-\boldsymbol{n}, 0}$ and $\alpha\geq 0$, we have $P_{\alpha}(\boldsymbol{n}, \hat{\boldsymbol{n}})-Q(\boldsymbol{n}, \hat{\boldsymbol{n}})\geq 0$.
For a given $\boldsymbol{p}\in \mathbb{S}^1_{-\boldsymbol{n}}, \boldsymbol{p}\neq -\boldsymbol{n}$, we know that $(\boldsymbol{p}\cdot \boldsymbol{n}^\perp)^2>0$. Therefore we can choose a sufficiently large but finite $\alpha_{\boldsymbol{p}}<\infty$, such that $P_{\alpha_{\boldsymbol{p}}}(\boldsymbol{n},\boldsymbol{p})-Q(\boldsymbol{n}, \boldsymbol{p})> 0$. And by the same argument, there exists an open set $\mathcal{O}_{\boldsymbol{p}, \alpha_{\boldsymbol{p}}}\subset \mathbb{S}^1$, such that $\boldsymbol{p}\in \mathcal{O}_{\boldsymbol{p}, \alpha_{\boldsymbol{p}}}$ and $P_{\alpha}(\boldsymbol{n}, \hat{\boldsymbol{n}})-Q(\boldsymbol{n}, \hat{\boldsymbol{n}})\geq 0, \forall \hat{\boldsymbol{n}}\in \mathcal{O}_{\boldsymbol{p}, \alpha_{\boldsymbol{p}}}, \alpha\geq \alpha_{\boldsymbol{p}}$. And we obtain an open cover for $\mathbb{S}_{-\boldsymbol{n}}^1$ as
\begin{equation}
    \mathbb{S}^1_{-\boldsymbol{n}}\subset \bigcup\limits_{\boldsymbol{p}\in \mathbb{S}^1_{-\boldsymbol{n}}} \mathcal{O}_{\boldsymbol{p}, \alpha_{\boldsymbol{p}}}.
\end{equation}
Since $\mathbb{S}^1_{-\boldsymbol{n}}$ is compact, by open cover theorem, there is a finite set of vectors $\boldsymbol{p}_1, \ldots, \boldsymbol{p}_M \in \mathbb{S}^1_{-\boldsymbol{n}}$, such that
\begin{equation}
    \mathbb{S}^1_{-\boldsymbol{n}}\subset \bigcup\limits_{i=1}^M \mathcal{O}_{\boldsymbol{p}_i, \alpha_{\boldsymbol{p}_i}}.
  \end{equation}
If we take $\alpha_2:=\max\limits_{1\leq i\leq M} \alpha_{\boldsymbol{p}_i}<\infty$, we have $P_{\alpha}(\boldsymbol{n}, \hat{\boldsymbol{n}})-Q(\boldsymbol{n}, \hat{\boldsymbol{n}})\geq 0, \forall \hat{\boldsymbol{n}}\in \mathbb{S}^1_{-\boldsymbol{n}}, \alpha\geq \alpha_2$, hence
\begin{equation}
    \infty>\max(\alpha_1, \alpha_2)\in \left\{\alpha\geq 0: P_{\alpha}(\boldsymbol{n}, \hat{\boldsymbol{n}})-Q(\boldsymbol{n}, \hat{\boldsymbol{n}})\geq 0,\ \forall \hat{\boldsymbol{n}}\in\mathbb{S}^1\right\}.
\end{equation}
Which means $k_0(\boldsymbol{n})=\inf \left\{\alpha\geq 0: P_{\alpha}(\boldsymbol{n}, \hat{\boldsymbol{n}})-Q(\boldsymbol{n}, \hat{\boldsymbol{n}})\geq 0, \forall \hat{\boldsymbol{n}}\in\mathbb{S}^1\right\}<\infty$. \qed
\end{proof}
\begin{remark}The $\hat{\boldsymbol{n}}\in \mathbb{S}^1_{\boldsymbol{n}}$ part only requires inequality \eqref{temp, regularity of gamma}, thus condition $\gamma(\boldsymbol{p})\in C^2(\mathbb{R}^2\setminus \{\boldsymbol{0}\})$ can be relaxed to $\gamma(\boldsymbol{p})$ is piecewise $C^2(\mathbb{R}^2\setminus \{\boldsymbol{0}\})$. And the condition $3\gamma(\boldsymbol{n})>\gamma(-\boldsymbol{n})$ is to ensure the existence of $\mathcal{O}_{-\boldsymbol{n}, 0}$, which suggests we may find a larger $\alpha_{-\boldsymbol{n}}>0$ such that $\mathcal{O}_{-\boldsymbol{n}, \alpha_{-\boldsymbol{n}}}$ exists for $\gamma(\boldsymbol{n})$ satisfies $3\gamma(\boldsymbol{n})\geq \gamma(-\boldsymbol{n})$. And by the same argument, we can show such $\mathcal{O}_{-\boldsymbol{n}, \alpha_{-\boldsymbol{n}}}$ exists if and only if $\boldsymbol{\xi}(\boldsymbol{n}_i)=\gamma(\boldsymbol{n}_i)\boldsymbol{n}_i \iff 3\gamma(\boldsymbol{n}_i)=\gamma(-\boldsymbol{n}_i)$.
\end{remark}

By the existence of $k_0(\boldsymbol{n})$, once the $\gamma(\boldsymbol{n})$ is given, the minimal stabilizing function $k_0(\boldsymbol{n})$ is then determined, i.e. there is a map from $\gamma(\boldsymbol{n})$ to $k_0(\boldsymbol{n})$. Similar to the proof when $\gamma(\boldsymbol{n})$ is symmetric, i.e. $\gamma(\boldsymbol{n})=\gamma(-\boldsymbol{n})$  in \cite{bao2021structurepreserving,bao2022volume}, we can show such map is a sub-linear with respect to $\gamma(\boldsymbol{n})$ when
it satisfies \eqref{es condition on gamma}.

\begin{lemma}[positive homogeneity and subadditivity]
Let $\gamma_1(\boldsymbol{n}), \gamma_2(\boldsymbol{n})$ and $\gamma_3(\boldsymbol{n})$ be three functions satisfying \eqref{es condition on gamma} with minimal stabilizing functions $k_1(\boldsymbol{n}), k_2(\boldsymbol{n})$ and  $k_3(\boldsymbol{n})$, respectively, then we have

(i) if $\gamma_2(\boldsymbol{n})=c \gamma_1(\boldsymbol{n})$, where $c>0$ is a positive number, then $k_2(\boldsymbol{n})=ck_1(\boldsymbol{n})$; and

(ii) if $\gamma_3(\boldsymbol{n})=\gamma_1(\boldsymbol{n})+\gamma_2(\boldsymbol{n})$, then $k_3(\boldsymbol{n})\leq k_1(\boldsymbol{n})+k_2(\boldsymbol{n})$.
\end{lemma}

The proof is similar to \cite[Lemma 4.4]{bao2021symmetrized}, and is omitted for brevity.

\subsection{Proof of the energy dissipation in
\eqref{geo properties, full, sd}}

To prove the main result, we first need the following lemma:
\begin{lemma}\label{lemma: local est}
Suppose $\boldsymbol{h}, \hat{\boldsymbol{h}}$ are two non-zero vectors in $\mathbb{R}^2$ and $\boldsymbol{n}=-\frac{\boldsymbol{h}^\perp}{|\boldsymbol{h}|}, \hat{\boldsymbol{n}}=-\frac{\hat{\boldsymbol{h}}^\perp}{|\hat{\boldsymbol{h}}|}$ to be the corresponding unit normal vectors. Then for any $k(\boldsymbol{n})\geq k_0(\boldsymbol{n})$, the following inequality holds
\begin{equation}\label{energy diff between two steps, local lemma}
    \frac{1}{|\boldsymbol{h}|}\left(\boldsymbol{G}_k(\boldsymbol{n})
    \hat{\boldsymbol{h}}\right)\cdot (\hat{\boldsymbol{h}}-\boldsymbol{h})\geq|\hat{\boldsymbol{h}}|\gamma(\hat{\boldsymbol{n}})-|\boldsymbol{h}|\,\gamma(\boldsymbol{n}).
\end{equation}
\end{lemma}
\begin{proof}
By applying the definition of $\boldsymbol{G}_k(\boldsymbol{n})$ in \eqref{def of G_k} and $P_{\alpha}(\boldsymbol{n}, \hat{\boldsymbol{n}})$ in \eqref{def of F_k}, and noticing $\hat{\boldsymbol{h}}=\hat{\boldsymbol{n}}^\perp |\hat{\boldsymbol{h}}|, \boldsymbol{n}\cdot \hat{\boldsymbol{n}}^\perp=-\hat{\boldsymbol{n}}\cdot \boldsymbol{n}^\perp$, the term $\frac{1}{|\boldsymbol{h}|}\left(\boldsymbol{G}_k(\boldsymbol{n})\hat{\boldsymbol{h}}\right)\cdot \hat{\boldsymbol{h}}$ can be simplified as
\begin{align}\label{first part, Pk}
    &\frac{1}{|\boldsymbol{h}|}\left(\boldsymbol{G}_k(\boldsymbol{n})\hat{\boldsymbol{h}}\right)\cdot \hat{\boldsymbol{h}}\nonumber\\
    &=\frac{1}{|\boldsymbol{h}|}\left[\left((\gamma(\boldsymbol{n})I_2+k(\boldsymbol{n})
    \boldsymbol{n}(\boldsymbol{n})^T)+(\boldsymbol{\xi}
    (\boldsymbol{n})^T-\boldsymbol{n}(\boldsymbol{\xi})^T)
    \right)\hat{\boldsymbol{h}}\right]\cdot \hat{\boldsymbol{h}}\nonumber\\
    &=\frac{1}{|\boldsymbol{h}|}\left[\left(\gamma(\boldsymbol{n})I_2+k(\boldsymbol{n})
    \boldsymbol{n}(\boldsymbol{n})^T
    \right)\hat{\boldsymbol{h}}\right]\cdot \hat{\boldsymbol{h}}\nonumber\\
    &=\frac{1}{|\boldsymbol{h}|}\left(\gamma(\boldsymbol{n})|\hat{\boldsymbol{h}}|^2+k(\boldsymbol{n})(\boldsymbol{n}\cdot \hat{\boldsymbol{h}})^2\right)\nonumber\\
    &=\frac{1}{|\boldsymbol{h}|}\left(\gamma(\boldsymbol{n})|\hat{\boldsymbol{h}}|^2+k(\boldsymbol{n})(\hat{\boldsymbol{n}}\cdot (\boldsymbol{n})^\perp)^2|\hat{\boldsymbol{h}}|^2\right)=\frac{|\hat{\boldsymbol{h}}|^2}{4|\boldsymbol{h}|\gamma(\boldsymbol{n})}P^2_{k(\boldsymbol{n})}(\boldsymbol{n}, \hat{\boldsymbol{n}}).
    \end{align}
Similarly, by applying the definition of $\boldsymbol{G}_k(\boldsymbol{n})$ in \eqref{def of G_k} and $Q(\boldsymbol{n}, \hat{\boldsymbol{n}})$ in \eqref{def of F}, and noticing $\boldsymbol{h}=\boldsymbol{n}^\perp |\boldsymbol{h}|, \hat{\boldsymbol{h}}=\hat{\boldsymbol{n}}^\perp |\hat{\boldsymbol{h}}|,\boldsymbol{h}\cdot \hat{\boldsymbol{h}}= \boldsymbol{n}\cdot \hat{\boldsymbol{n}}|\boldsymbol{h}|\,|\hat{\boldsymbol{h}}|, \boldsymbol{n}\cdot \hat{\boldsymbol{n}}^\perp=-\hat{\boldsymbol{n}}\cdot \boldsymbol{n}^\perp$, the term $\frac{1}{|\boldsymbol{h}|}\left(\boldsymbol{G}_k(\boldsymbol{n})\hat{\boldsymbol{h}}\right)\cdot \boldsymbol{h}$ can be simplified as
\begin{align}\label{second part, Q}
    &\frac{1}{|\boldsymbol{h}|}\left(\boldsymbol{G}_k(\boldsymbol{n})\hat{\boldsymbol{h}}\right)\cdot \boldsymbol{h}=\frac{1}{|\boldsymbol{h}|}\left(\boldsymbol{G}_k^T(\boldsymbol{n})\boldsymbol{h}\right)\cdot \hat{\boldsymbol{h}}\nonumber\\
    &=\frac{1}{|\boldsymbol{h}|}\left[\left((\gamma(\boldsymbol{n})I_2+\boldsymbol{n}
    (\boldsymbol{\xi})^T)+(-\boldsymbol{\xi}(\boldsymbol{n})^T
    +k(\boldsymbol{n})\boldsymbol{n}(\boldsymbol{n})^T)\right)
    \boldsymbol{h}\right]\cdot \hat{\boldsymbol{h}}\nonumber\\
    &=\frac{1}{|\boldsymbol{h}|}\left[\left(\gamma(\boldsymbol{n})I_2+\boldsymbol{n}
    (\boldsymbol{\xi})^T\right)
    \boldsymbol{h}\right]\cdot \hat{\boldsymbol{h}}\nonumber\\
    &=\frac{1}{|\boldsymbol{h}|}\left[\gamma(\boldsymbol{n})(\boldsymbol{h}\cdot \hat{\boldsymbol{h}})+(\boldsymbol{\xi}\cdot \boldsymbol{h})(\boldsymbol{n}\cdot \hat{\boldsymbol{h}})\right]\nonumber\\
    &=|\hat{\boldsymbol{h}}|\,\left
    (\gamma(\boldsymbol{n})(\boldsymbol{n}\cdot \hat{\boldsymbol{n}})+(\boldsymbol{\xi}\cdot \boldsymbol{n}^\perp)(\boldsymbol{n}\cdot \hat{\boldsymbol{n}}^\perp)\right)\nonumber\\
    &=|\hat{\boldsymbol{h}}|\,\left
    (\gamma(\boldsymbol{n})(\boldsymbol{n}\cdot \hat{\boldsymbol{n}})-(\boldsymbol{\xi}\cdot \boldsymbol{n}^\perp)(\hat{\boldsymbol{n}}\cdot \boldsymbol{n}^\perp)\right)=|\hat{\boldsymbol{h}}|\,\left(Q(\boldsymbol{n}, \hat{\boldsymbol{n}})-\gamma(\hat{\boldsymbol{n}})\right).
    \end{align}

Finally, combining the definition of $k_0(\boldsymbol{n})$ \eqref{def of k0}, \eqref{first part, Pk}, \eqref{second part, Q}, and the fact $\frac{a^2}{4|\boldsymbol{h}|\gamma(\boldsymbol{n})}\geq -ab+|\boldsymbol{h}|\gamma(\boldsymbol{n})b^2$ yields that
\begin{align}
    \frac{1}{|\boldsymbol{h}|}\left(\boldsymbol{G}_k(\boldsymbol{n})
    \hat{\boldsymbol{h}}\right)\cdot (\hat{\boldsymbol{h}}-\boldsymbol{h})&=\frac{|\hat{\boldsymbol{h}}|^2}{4|\boldsymbol{h}|\gamma(\boldsymbol{n})}P^2_{k(\boldsymbol{n})}(\boldsymbol{n}, \hat{\boldsymbol{n}})-|\hat{\boldsymbol{h}}|\,\left(Q(\boldsymbol{n}, \hat{\boldsymbol{n}})-\gamma(\hat{\boldsymbol{n}})\right)\nonumber\\
    &\geq -|\hat{\boldsymbol{h}}|P_{k(\boldsymbol{n})}+|\boldsymbol{h}|\gamma(\boldsymbol{n})-|\hat{\boldsymbol{h}}|\,\left(P_{k(\boldsymbol{n})}(\boldsymbol{n}, \hat{\boldsymbol{n}})-\gamma(\hat{\boldsymbol{n}})\right)\nonumber\\
    &=|\hat{\boldsymbol{h}}|\gamma(\hat{\boldsymbol{n}})-|\boldsymbol{h}|\,\gamma(\boldsymbol{n}),
    \end{align}
which validates \eqref{energy diff between two steps, local lemma}.\qed
\end{proof}
Now we can prove the energy dissipation part in our main result Theorem \ref{thm: main}.
\begin{proof}
    The key point of the proof is to establish the following energy estimation
    \begin{equation}\label{energy diff between two steps}
        \Bigl( \boldsymbol{G}_k(\boldsymbol{n}^m)\partial_s \boldsymbol{X}^{m+1},\partial_s(\boldsymbol{X}^{m+1}
        -\boldsymbol{X}^m)\Bigr)_{\Gamma^m}^h\geq W^{m+1}-W^{m},\quad m\ge0.
    \end{equation}

    For any $1\leq j\leq N$, take $\boldsymbol{h}=\boldsymbol{h}_j^m, \hat{\boldsymbol{h}}=\boldsymbol{h}_j^{m+1}$ in Lemma \ref{lemma: local est}, we know that  $\boldsymbol{n}=-\frac{\boldsymbol{h}^\perp}{|\boldsymbol{h}|}=\boldsymbol{n}^m_j, \hat{\boldsymbol{n}}=\boldsymbol{n}_j^{m+1}$, and the following inequality holds
    \begin{equation}\label{energy diff between two steps, local}
    \frac{1}{|\boldsymbol{h}_j^m|}\left(\boldsymbol{G}_k(\boldsymbol{n}_j^m)
    \boldsymbol{h}_j^{m+1}\right)\cdot (\boldsymbol{h}_j^{m+1}-\boldsymbol{h}_j^m)\geq |\boldsymbol{h}_j^{m+1}|\gamma(\boldsymbol{n}_j^{m+1})-|\boldsymbol{h}_j^m|\,\gamma(\boldsymbol{n}_j^m).
    \end{equation}
    Take summation over $1\leq j\leq N$ for \eqref{energy diff between two steps, local} and notice \eqref{mass lumped inner product} and \eqref{fully discrete area and energy, energy}, we have
    \begin{align*}
    &\Bigl( \boldsymbol{G}_k(\boldsymbol{n}^m)\partial_s \boldsymbol{X}^{m+1},\partial_s(\boldsymbol{X}^{m+1}-\boldsymbol{X}^m)\Bigr)_{\Gamma^m}^h\\
    &=\sum_{j=1}^N\left[|\boldsymbol{h}_j^m|\left(\boldsymbol{G}_k(\boldsymbol{n}_j^m)\frac{\boldsymbol{h}_j^{m+1}}{|\boldsymbol{h}_j^m|}\right)\cdot \frac{\boldsymbol{h}_j^{m+1}-\boldsymbol{h}_j^m}{|\boldsymbol{h}_j^m|}\right]\\
    &=\sum_{j=1}^N\left[\frac{1}{|\boldsymbol{h}_j^m|}\left(\boldsymbol{G}_k(\boldsymbol{n}_j^m)\boldsymbol{h}_j^{m+1}\right)\cdot (\boldsymbol{h}_j^{m+1}-\boldsymbol{h}_j^m)\right]\\
    &\geq\sum_{j=1}^N\left[|\boldsymbol{h}_j^{m+1}|\gamma(\boldsymbol{n}_j^{m+1})-|\boldsymbol{h}_j^m|\,\gamma(\boldsymbol{n}_j^m)\right]\\
    &=W^{m+1}-W^m, \qquad \forall m=0, 1,\ldots
    \end{align*}
    which proves the energy estimation in \eqref{energy diff between two steps}.

    Finally, take $\varphi^h=\mu^{m+1}$ in \eqref{eqn:2dfull1 sd} and $\boldsymbol{\omega}^h=\boldsymbol{X}^{m+1}-\boldsymbol{X}^m$ in \eqref{eqn:2dfull2 sd}, we have
    \begin{align*}
        0&\geq -\tau \Bigl( \partial_s\mu^{m+1},\partial_s\mu^{m+1}\Bigr)_{\Gamma^m}^h\\
        &=\Bigl(\boldsymbol{n}^{m+\frac{1}{2}}\cdot \left(\boldsymbol{X}^{m+1}-\boldsymbol{X}^m\right),\mu^{m+1}\Bigr)_{\Gamma^m}^h\\
        &=\Bigl( \boldsymbol{G}_k(\boldsymbol{n}^m)\partial_s \boldsymbol{X}^{m+1},\partial_s(\boldsymbol{X}^{m+1}-\boldsymbol{X}^m)\Bigr)_{\Gamma^m}^h\geq W^{m+1}-W^m,
    \end{align*}
    this is true for any $m$; therefore, the energy $W^m$ decreases monotonically, and the proof is completed. \qed
\end{proof}
\begin{remark}
The condition $\gamma(\boldsymbol{p})\in C^2(\mathbb{R}^2\setminus \{\boldsymbol{0}\})$ in \eqref{es condition on gamma} is natural, but $3\gamma(\boldsymbol{n})>\gamma(-\boldsymbol{n})$ looks quite complicated and seems not very sharp. However, the proof shows the condition $3\gamma(\boldsymbol{n})>\gamma(-\boldsymbol{n})$ is indeed natural! To see this, inequality \eqref{energy diff between two steps, local lemma} in Lemma \ref{lemma: local est} is essential in showing the energy estimate \eqref{energy diff between two steps}. And if we take $\hat{\boldsymbol{h}}=-\boldsymbol{h}$ in Lemma \ref{lemma: local est}, then $\hat{\boldsymbol{n}}=-\boldsymbol{n}$, and the inequality \eqref{energy diff between two steps, local lemma} becomes
\begin{equation}
    2\gamma(\boldsymbol{n})\,|\boldsymbol{h}|\geq |\boldsymbol{h}|\,\gamma(-\boldsymbol{n})-\gamma(\boldsymbol{n})\,
    |\boldsymbol{h}|,
\end{equation}
which means $3\gamma(\boldsymbol{n})\geq \gamma(-\boldsymbol{n})$. Our sufficient energy stable condition \eqref{es condition on gamma} just replaces $\geq$ with $>$, thus is natural and \textbf{almost necessary}.
\end{remark}
\section{Extension to other anisotropic geometric flows}
In fact, the energy stable condition on $\gamma(\boldsymbol{n})$ in
\eqref{es condition on gamma}, the definition of $\boldsymbol{G}_k(\boldsymbol{n})$ in \eqref{def of G_k}, the alternative expression for $\mu$ in \eqref{alt def of mu}, and the definition of $k_0(\boldsymbol{n})$ in \eqref{def of k0} are independent of the anisotropic surface diffusion flow. Thus these definitions and even the proof of energy stability can be directly extended to other anisotropic geometric flows.

\subsection{For anisotropic curvature flow}
Similar to \eqref{eqn:conserve continuous sd}, for the anisotropic curvature flow in \eqref{eqn: aniso geo flow}, we have a conservative geometric PDE as
\begin{subequations}
\label{eqn:conserve continuous mc}
\begin{align}
\label{eqn:conserve1 continuous mc}
&\boldsymbol{n}\cdot \partial_t \boldsymbol{X}=-\mu,\\[0.5em]
\label{eqn:conserve2 continuous mc}
&\mu\boldsymbol{n}=-\partial_s(\boldsymbol{G}_k(\boldsymbol{n})\partial_s \boldsymbol{X}).
\end{align}
\end{subequations}

Suppose the initial curve $\boldsymbol{X}(\cdot, 0)=(x(\cdot, 0), y(\cdot, 0))^T:=\Gamma_0\in [H^1(\mathbb{T})]^2$ and the initial weighted curvature $\mu(\cdot, 0):=\mu_0(\cdot)\in H^1(\mathbb{T})$. Based on the conservative form \eqref{eqn:conserve continuous mc}, the variational formulation for anisotropic curvature flow is as follows: For any $t>0$, find the solution $(\boldsymbol{X}(\cdot, t), \mu(\cdot, t))\in [H^1(\mathbb{T})]^2\times H^1(\mathbb{T})$ satisfying
\begin{subequations}
\label{eqn:weak continuous mc}
\begin{align}
\label{eqn:weak1 continuous mc}
&\Bigl(\boldsymbol{n}\cdot\partial_{t}\boldsymbol{X},
\varphi\Bigr)_{\Gamma(t)}+\Bigl(\mu,
\varphi\Bigr)_{\Gamma(t)}=0,\qquad\forall \varphi\in H^1(\mathbb{T}),\\[0.5em]
\label{eqn:weak2 continuous mc}
&\Bigl(\mu\boldsymbol{n},\boldsymbol{\omega}\Bigr)_{\Gamma(t)}-\Bigl( \boldsymbol{G}_k(\boldsymbol{n})\partial_s \boldsymbol{X},\partial_s\boldsymbol{\omega}\Bigr)_{\Gamma(t)}= 0,\quad\forall\boldsymbol{\omega}\in [H^1(\mathbb{T})]^2.
\end{align}
\end{subequations}
And the SP-PFEM for the anisotropic curvature flow
\eqref{eqn:weak continuous mc} is as follows: Suppose the initial approximation $\Gamma^0(\cdot)\in [\mathbb{K}^h]^2$ is given by $\boldsymbol{X}^0(\rho_j)=\boldsymbol{X}_0(\rho_j), \forall 0\leq j\leq N$, then for any $m=0, 1, 2, \ldots$, find the solution $(\boldsymbol{X}^m(\cdot), \mu^m(\cdot))\in [\mathbb{K}^h]^2\times \mathbb{K}^h$, such that
\begin{subequations}
\label{eqn:2dfull mc}
\begin{align}
\label{eqn:2dfull1 mc}
&\Bigl(\boldsymbol{n}^{m+\frac{1}{2}}\cdot\frac{\boldsymbol{X}^{m+1}-\boldsymbol{X}^m}{\tau},
\varphi^h\Bigr)_{\Gamma^m}^h+\Bigl(\mu^{m+1},
\varphi^h\Bigr)_{\Gamma^m}^h=0,\quad\forall \varphi^h\in \mathbb{K}^h,\\
\label{eqn:2dfull2 mc}
&\Bigl(\mu^{m+1}\boldsymbol{n}^{m+\frac{1}{2}},\boldsymbol{\omega}^h\Bigr)_{\Gamma^m}^h-\Bigl( \boldsymbol{G}_k(\boldsymbol{n}^m)\partial_s \boldsymbol{X}^{m+1},\partial_s\boldsymbol{\omega}^h\Bigr)_{\Gamma^m}^h=0,
\quad\forall\boldsymbol{\omega}^h\in[\mathbb{K}^h]^2.
\end{align}
\end{subequations}

    For the SP-PFEM \eqref{eqn:2dfull mc}, we have

\begin{theorem}[structure-preserving]\label{coro: main mc} Suppose $\gamma(\boldsymbol{n})$ satisfies \eqref{es condition on gamma} and
take a stabilizing function $k(\boldsymbol{n})\geq k_0(\boldsymbol{n})$,
then the SP-PFEM \eqref{eqn:2dfull mc} preserves area decay rate and energy dissipation, i.e.,
\begin{equation}\label{geo properties, full, mc}
    \frac{A^{m+1}-A^m}{\tau}=-\Bigl(\mu^{m+1},1\Bigr)_{\Gamma^m}^h, \quad W^{m+1}\leq W^m\leq \ldots\le  W^0, \qquad \forall m\geq 0.
    \end{equation}
\end{theorem}
\begin{proof}
    From \cite[Theorem 2.1]{bao2021structurepreserving}, we know that
    \begin{equation}\label{area difference}
        A^{m+1}-A^m=\Bigl(\boldsymbol{n}^{m+\frac{1}{2}}\cdot(\boldsymbol{X}^{m+1}-\boldsymbol{X}^m),1\Bigr)_{\Gamma^m}^h.
    \end{equation}
    Thus by taking $\varphi^h\equiv 1\in \mathbb{K}^h$  in  \eqref{eqn:2dfull1 mc}, we know that
    \begin{equation}
        \frac{A^{m+1}-A^m}{\tau}=\Bigl(\boldsymbol{n}^{m+\frac{1}{2}}\cdot\frac{\boldsymbol{X}^{m+1}-\boldsymbol{X}^m}{\tau},1\Bigr)_{\Gamma^m}^h=-\Bigl(\mu^{m+1},1\Bigr)_{\Gamma^m}^h,
    \end{equation}
    which is the desired decay rate in \eqref{geo properties, full, mc}.

    For energy dissipation, we have already known that \eqref{energy diff between two steps} is true. Taking  $\varphi^h=\mu^{m+1}$ in \eqref{eqn:2dfull1 mc} and $\boldsymbol{\omega}^h=\boldsymbol{X}^{m+1}-\boldsymbol{X}^m$ in \eqref{eqn:2dfull2 mc}, we know that
    \begin{align*}
        0&\geq -\tau \Bigl(\mu^{m+1},\mu^{m+1}\Bigr)_{\Gamma^m}^h\\
        &=\Bigl( \boldsymbol{G}_k(\boldsymbol{n}^m)\partial_s \boldsymbol{X}^{m+1},\partial_s(\boldsymbol{X}^{m+1}-\boldsymbol{X}^m)\Bigr)_{\Gamma^m}^h\geq W^{m+1}-W^m.
    \end{align*}
    Hence we complete the proof. \qed
\end{proof}

\subsection{For area-conserved anisotropic curvature flow}
Similarly, for the area-conserved anisotropic curvature flow in
\eqref{eqn: aniso geo flow}, the conservative geometric PDE is given as
\begin{subequations}
\label{eqn:conserve continuous acmc}
\begin{align}
\label{eqn:conserve1 continuous acmc}
&\boldsymbol{n}\cdot \partial_t \boldsymbol{X}=-\mu+\lambda(t),\\[0.5em]
\label{eqn:conserve2 continuous acmc}
&\mu\boldsymbol{n}=-\partial_s(\boldsymbol{G}_k(\boldsymbol{n})\partial_s \boldsymbol{X}),
\end{align}
\end{subequations}
where $\lambda(t)$ is given as \eqref{lambdcon} by replacing
$\lambda$ and $\Gamma$ by $\lambda(t)$ and $\Gamma(t)$, respectively.
And the variational formulation can be derived in a similar way.

In order to design a structure-preserving full discretization, we need to properly discretize $\lambda(t)$. Denote $\lambda^{m+\frac{1}{2}}$ with respect to $\Gamma^m$ as
\begin{equation}
    \lambda^{m+\frac{1}{2}}:=\frac{\left(\mu^{m+1}, 1\right)_{\Gamma^m}^h}{\left(1, 1\right)_{\Gamma^m}^h}.
\end{equation}
By adopting this $\lambda^{m+\frac{1}{2}}$, the SP-PFEM for the
 area-conserved anisotropic curvature flow in \eqref{eqn: aniso geo flow} is as follows: Suppose the initial approximation $\Gamma^0(\cdot)\in [\mathbb{K}^h]^2$ is given by $\boldsymbol{X}^0(\rho_j)=\boldsymbol{X}_0(\rho_j), \forall 0\leq j\leq N$; for any $m=0, 1, 2, \ldots$, find the solution $(\boldsymbol{X}^m(\cdot), \mu^m(\cdot))\in [\mathbb{K}^h]^2\times \mathbb{K}^h$, such that
\begin{subequations}
\label{eqn:2dfull acmc}
\begin{align}
\label{eqn:2dfull1 acmc}
&\Bigl(\boldsymbol{n}^{m}\cdot\frac{\boldsymbol{X}^{m+1}-\boldsymbol{X}^m}{\tau},
\varphi^h\Bigr)_{\Gamma^m}^h+\Bigl(\mu^{m+1}-\lambda^{m+\frac{1}{2}},
\varphi^h\Bigr)_{\Gamma^m}^h=0,\quad\forall \varphi^h\in \mathbb{K}^h,\\
\label{eqn:2dfull2 acmc}
&\Bigl(\mu^{m+1}\boldsymbol{n}^{m},\boldsymbol{\omega}^h\Bigr)_{\Gamma^m}^h-\Bigl( \boldsymbol{G}_k(\boldsymbol{n}^m)\partial_s \boldsymbol{X}^{m+1},\partial_s\boldsymbol{\omega}^h\Bigr)_{\Gamma^m}^h=0,
\quad\forall\boldsymbol{\omega}^h\in[\mathbb{K}^h]^2.
\end{align}
\end{subequations}

For the above SP-PFEM \eqref{eqn:2dfull acmc}, we have

\begin{theorem}[structure-preserving]\label{coro: main acmc} Suppose $\gamma(\boldsymbol{n})$ satisfies \eqref{es condition on gamma} and
take a finite stabilizing function $k(\boldsymbol{n})\geq k_0(\boldsymbol{n})$, then the SP-PFEM \eqref{eqn:2dfull mc} is
structure-preserving, i.e.,
\begin{equation}\label{geo properties, full, acmc}
A^{m+1}\equiv A^0, \quad W^{m+1}\leq W^m\leq \ldots\le W^0, \qquad \forall m\geq 0.
\end{equation}
\end{theorem}

\begin{proof}
For the area conservation,
taking $\varphi^h\equiv 1$ in \eqref{eqn:2dfull1 acmc} yields that
\begin{align*}
    \Bigl(\boldsymbol{n}^{m+\frac{1}{2}}\cdot(\boldsymbol{X}^{m+1}-\boldsymbol{X}^m),1\Bigr)_{\Gamma^m}^h&=-\tau \Bigl(\mu^{m+1}-\lambda^{m+\frac{1}{2}},1\Bigr)_{\Gamma^m}^h\\
    &=-\tau \Bigl(\mu^{m+1},1\Bigr)_{\Gamma^m}^h+\tau \lambda^{m+\frac{1}{2}}\Bigl(1, 1\Bigr)_{\Gamma^m}^h \\
    &=-\tau \Bigl(\mu^{m+1},1\Bigr)_{\Gamma^m}^h+\tau \frac{\left(\mu^{m+1}, 1\right)_{\Gamma^m}^h}{\left(1, 1\right)_{\Gamma^m}^h}\Bigl(1, 1\Bigr)_{\Gamma^m}^h\\
    &=0,\qquad m\ge0.
\end{align*}
By noting \eqref{geo properties, full, mc}, we deduce that $A^{m+1}-A^m=0$, which shows area conservation \eqref{fully discrete area and energy, area}.

For energy dissipation, by Cauchy-Schwarz inequality, we have
\begin{align*}
    \Bigl(\lambda^{m+\frac{1}{2}},\mu^{m+1}\Bigr)_{\Gamma^m}^h&=\lambda^{m+\frac{1}{2}}\Bigl(1,\mu^{m+1}\Bigr)_{\Gamma^m}^h\\
    &=\frac{1}{\left(1, 1\right)_{\Gamma^m}^h}\left(\Bigl(1,\mu^{m+1}\Bigr)_{\Gamma^m}^h\right)^2\\
    &\leq \frac{1}{\left(1, 1\right)_{\Gamma^m}^h}
    \Bigl(1,1\Bigr)_{\Gamma^m}^h\, \Bigl(\mu^{m+1},\mu^{m+1}\Bigr)_{\Gamma^m}^h\\
    &=\Bigl(\mu^{m+1},\mu^{m+1}\Bigr)_{\Gamma^m}^h,\qquad m\ge0.
\end{align*}
Taking $\varphi^h=\mu^{m+1}$ in \eqref{eqn:2dfull1 acmc} and $\boldsymbol{\omega}^h=\boldsymbol{X}^{m+1}-\boldsymbol{X}^m$ in \eqref{eqn:2dfull2 acmc}, and adopting the energy estimation \eqref{energy diff between two steps} yields that
\begin{equation}
    W^{m+1}-W^m\leq -\tau\Bigl(\mu^{m+1}-\lambda^{m+\frac{1}{2}},
    \mu^{m+1}\Bigr)_{\Gamma^m}^h\leq 0,\qquad m\ge0,
\end{equation}
which implies the energy dissipation in
\eqref{geo properties, full, acmc}. \qed
\end{proof}

\section{Numerical results}
In this section, we present numerical experiments to illustrate the high performance of the proposed SP-PFEMs. The implementations and performances of the three SP-PFEMs are very similar. Thus in section 5.1, we only show test results of the SP-PFEM \eqref{eqn:2dfull sd} for the anisotropic surface diffusion. The morphological evolutions for three anisotropic geometric flows are shown in section 5.2.

To compute the minimal stabilizing function $k_0(\boldsymbol{n})$, we first solve the optimization problem \eqref{def of k0} for $20$ uniformly distributed points $\boldsymbol{n}_j\in \mathbb{S}^1$, and then do linear interpolation for the intermediate points. In Newton's iteration, the tolerance value $\varepsilon$ is set as $10^{-12}$.

\subsection{Results for the anisotropic surface diffusion}Here we provide convergence tests to show the quadratic convergence rate in space and linear convergence rate in time. To this end, the time step $\tau$ is always chosen as $\tau=h^2$ except it is state otherwise. The distance between two closed curves $\Gamma_1, \Gamma_2$ is given by the manifold distance $M(\Gamma_1, \Gamma_2)$ in \cite{bao2020energy} as
\begin{equation}\label{manifold distance}
    M(\Gamma_1, \Gamma_2):=2|\Omega_1\cup \Omega_2|-|\Omega_1|-|\Omega_2|,
\end{equation}
where $\Omega_1, \Omega_2$ are the interior regions of $\Gamma_1, \Gamma_2$, respectively, and $|\Omega|$ denotes the area of $\Omega$. Let $\Gamma^m$ be the numerical approximation of $\Gamma^h(t=t_m:=m\tau)$ with mesh size $h$ and time step $\tau$, the numerical error is thus given as
\begin{equation}
    e^h(t)\Big|_{t=t_m}:=M(\Gamma^m, \Gamma(t_m)).
\end{equation}

\begin{figure}[t!]
\centering
\includegraphics[width=0.5\textwidth]{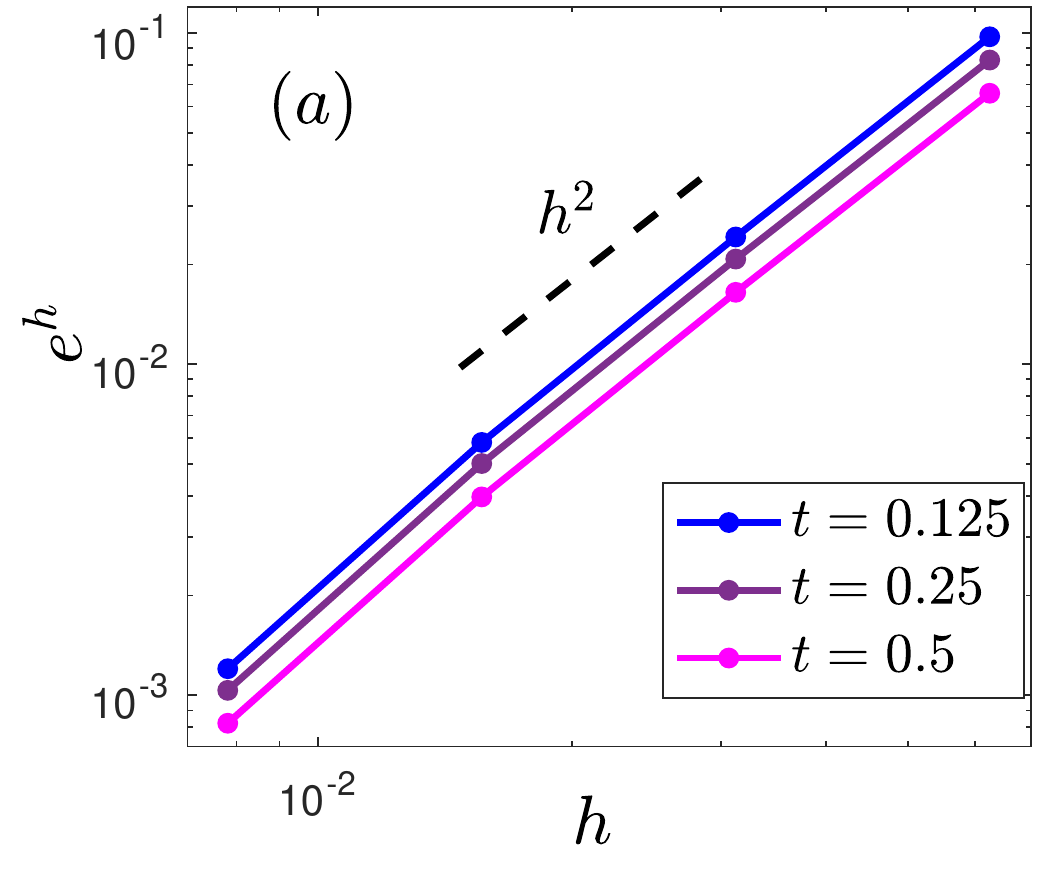}\includegraphics[width=0.5\textwidth]{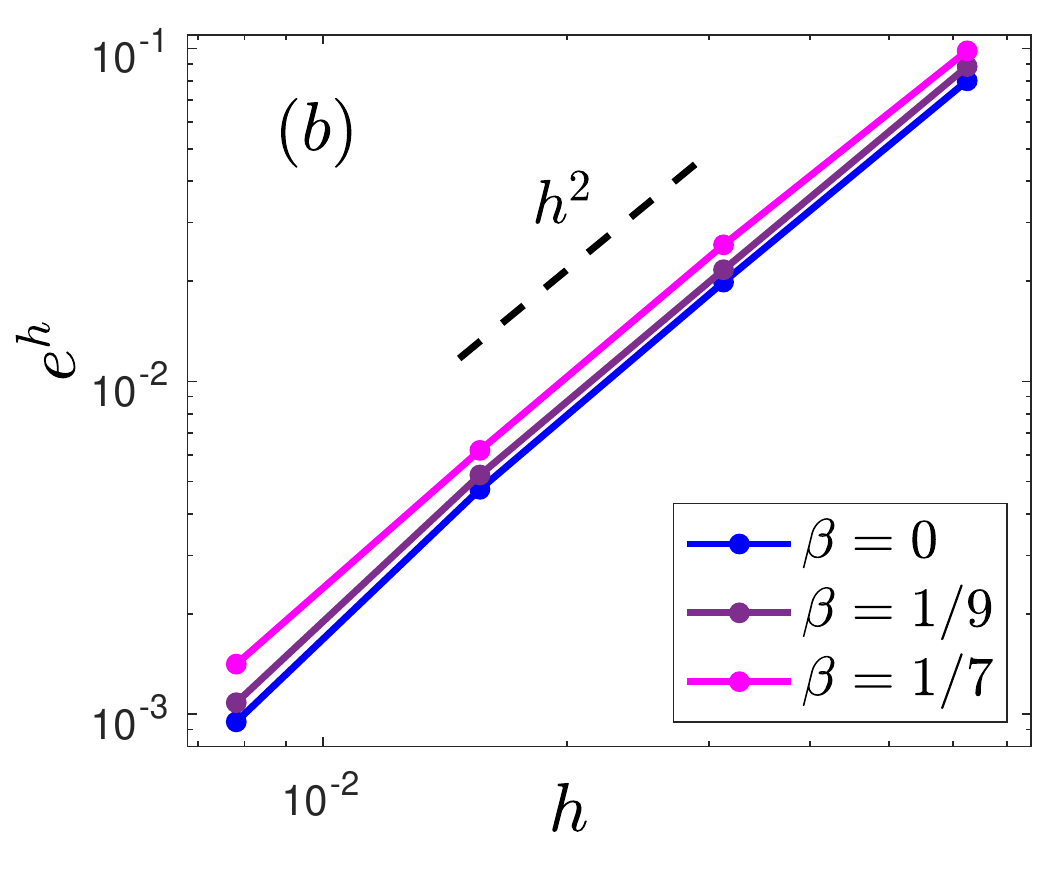}
\caption{Convergence rates of the SP-PFEM \eqref{eqn:2dfull sd} with $k(\boldsymbol{n})=k_0(\boldsymbol{n})$ for: (a) anisotropy in Case I at different times $t=0.125, 0.25, 0.5$; and (b) anisotropy in Case II at time $t=0.5$ with different $\beta$.}
\label{fig: convergent}
\end{figure}

To numerically test the energy stability, area conservation and good mesh quality, we introduce the following indicators: the normalized energy $\frac{W^h(t)}{W^h(0)}\Big|_{t=t_m}:=\frac{W^m}{W^0}$, the normalized area loss and the weighted mesh ratio
\begin{equation}
    \frac{\Delta A^h(t)}{A^h(0)}\Big|_{t=t_m}:=\frac{A^m-A^0}{A^0}, \qquad R^h_{\gamma}(t)\Big|_{t=t_m}:=\frac{\max\limits_{1\leq j\leq N} |\boldsymbol{h}_j^m|\,\gamma(\boldsymbol{n}_j^m)}{\min\limits_{1\leq j\leq N} |\boldsymbol{h}_j^m|\, \gamma(\boldsymbol{n}_j^m)}.
\end{equation}

In the following numerical tests, the initial curve $\Gamma_0$ is given as an ellipse with length $4$ and width $1$. The exact solution $\Gamma(t)$ is approximated by choosing $k(\boldsymbol{n})=k_0(\boldsymbol{n})$ with $h=2^{-8}$ and $\tau=2^{-16}$ in \eqref{eqn:2dfull sd}. Here are two typical anisotropic surface energies to be taken in our simulations:
\begin{itemize}
\item Case I: $\gamma(\boldsymbol{n})=\sqrt{\left(\frac{5}{2}+\frac{3}{2}\text{sgn}(n_1) \right)n_1^2+n_2^2}$ \cite{deckelnick2005computation};
\item Case II: $\gamma(\boldsymbol{n})=1+\beta \cos(3\theta) $ with $\boldsymbol{n}=(\sin\theta, -\cos\theta)^T$ and $|\beta|<1$ \cite{jiang2016solid}. It is weakly anisotropic when $|\beta|<\frac{1}{8}$, and otherwise it is strongly anisotropic.
    \end{itemize}

Fig. \ref{fig: convergent} presents the convergence rates of the proposed SP-PFEM \eqref{eqn:2dfull sd} at different times and with different anisotropic strengths $\beta$ under a fixed time $t=0.5$. It is apparent from this figure that the second-order convergence in space is independent of anisotropies and computation times, which indicates the convergence rate is very robust.

The time evolution of the normalized energy $\frac{W^h(t)}{W^h(0)}$ with different $h$, the normalized area loss $\frac{\Delta A^h(t)}{A^h(0)}$ and the number of Newton iterations with $h=2^{-7}$, and the weighted mesh quality $R^h_{\gamma}(t)$ with different $h$ are summarised in Figs.
\ref{fig: energy}-\ref{fig: mesh ratio}, respectively.

It can be seen from Figs. \ref{fig: energy}-\ref{fig: mesh ratio} that

(i) The normalized energy is monotonically decreasing when the surface energy satisfies the energy stable condition \eqref{es condition on gamma} (cf. Fig. \ref{fig: energy});

(ii) The normalized area loss is at $10^{-15}$ which is almost at the same order as the round-off error (cf. Fig. \ref{fig: energy}),
which confirms the area conservation in practical simulatoions;

(iii) Interestingly, the numbers of iterations in Newton's method are initially $2$ and finally $1$ (cf. Fig. \ref{fig: volume}). This finding suggests that although the proposed SP-PFEM \eqref{eqn:2dfull sd} is
full-implicit, but it can be solved very efficiently with a few iterations;

(iv) In Fig. \ref{fig: mesh ratio} there is a clear trend of convergence of the weighted mesh ratio $R_{\gamma}^h$. Moreover, in contrast to the symmetrized SP-PFEM for symmetric $\gamma(\boldsymbol{n})$ in \cite{bao2021symmetrized}, $R^h_{\gamma}$ keeps small even with the strongly anisotropy $\gamma(\boldsymbol{n})=1+\frac{1}{3}\cos 3\theta$.

\begin{figure}[htp!]
\centering
\includegraphics[width=1\textwidth]{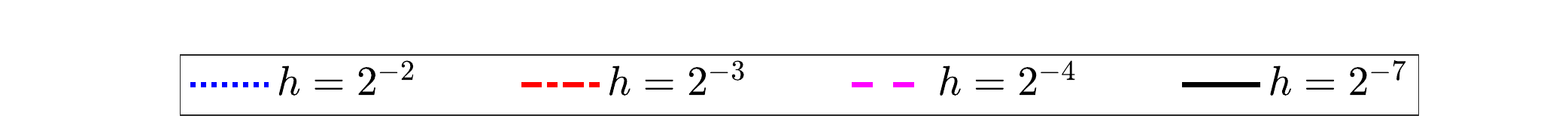}
\includegraphics[width=0.5\textwidth]{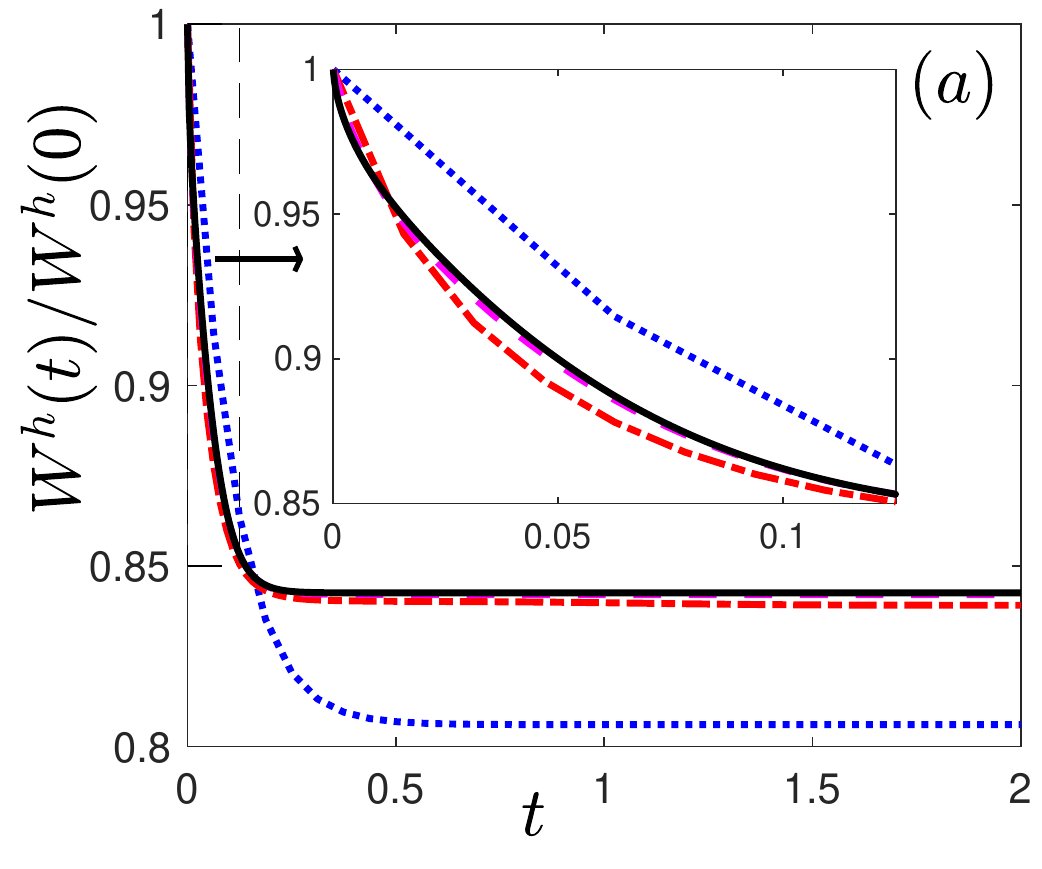}\includegraphics[width=0.5\textwidth]{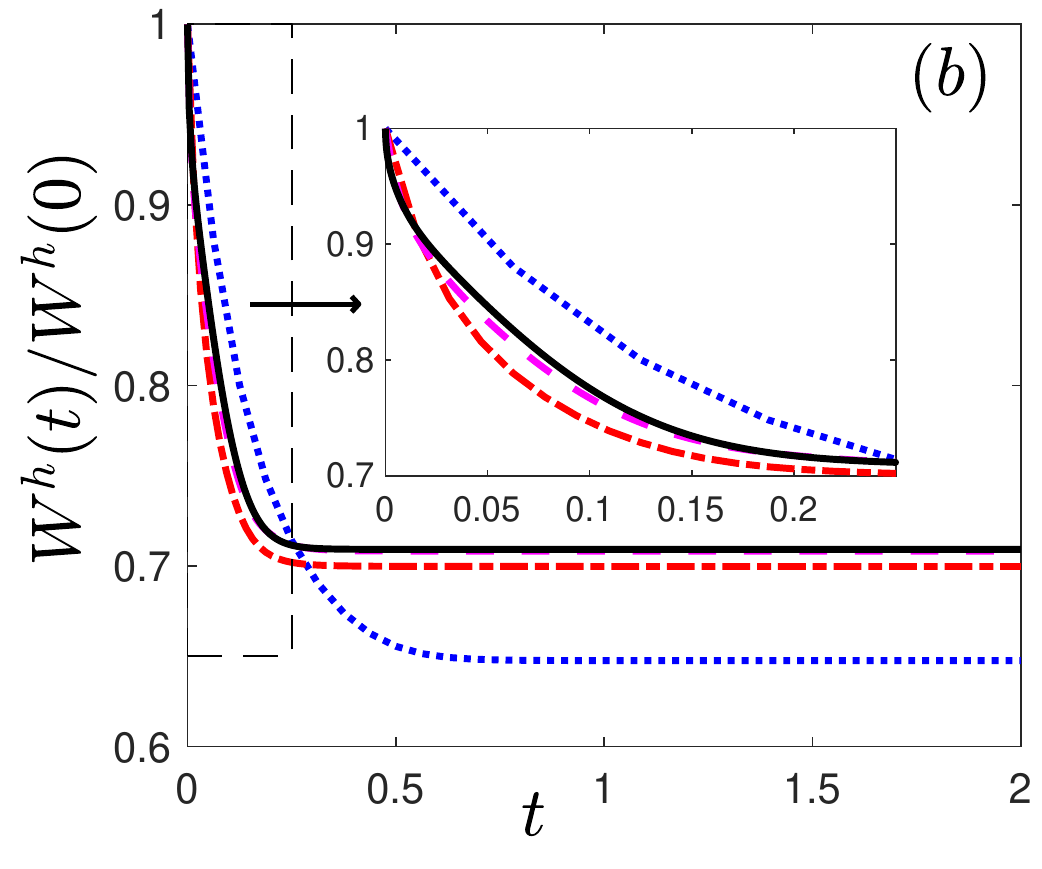}
\caption{Normalized energy of the SP-PFEM \eqref{eqn:2dfull sd} with $k(\boldsymbol{n})=k_0(\boldsymbol{n})$ and different $h$ for: (a) anisotropy in Case I; (b) anisotropy in Case II with $\beta=\frac{1}{3}$.}
\label{fig: energy}
\end{figure}
\begin{figure}[htp!]
\centering
\includegraphics[width=0.5\textwidth]{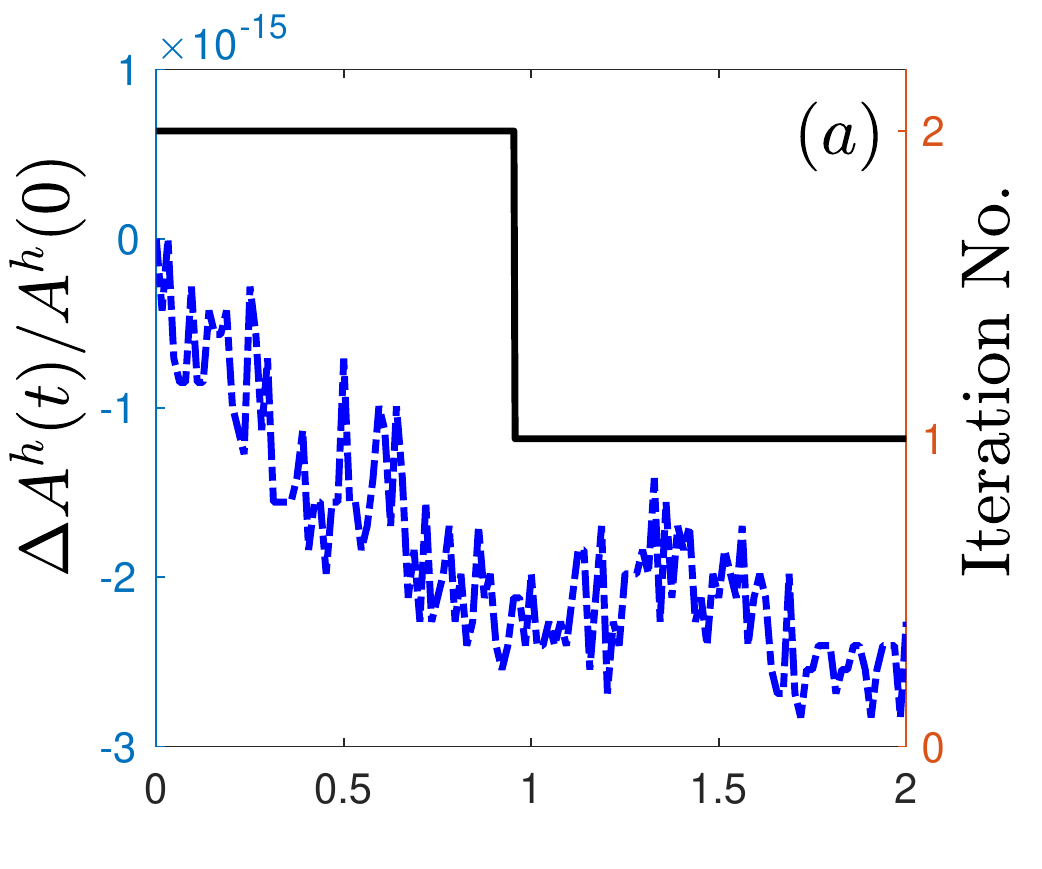}\includegraphics[width=0.5\textwidth]{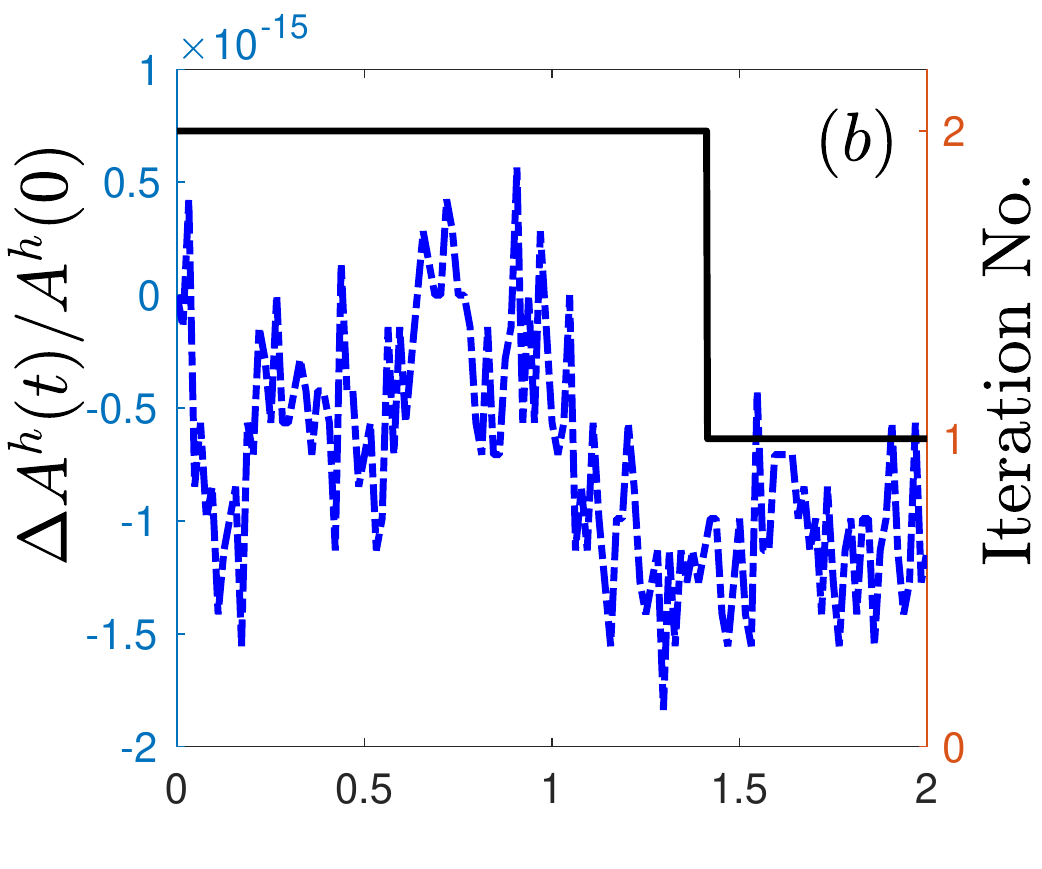}
\caption{Normalized area loss (blue dashed line) and iteration number (black line) of the SP-PFEM \eqref{eqn:2dfull sd} with $k(\boldsymbol{n})=k_0(\boldsymbol{n})$ and $h=2^{-7}, \tau=h^2$ for: (a) anisotropy in Case I; (b) anisotropy in Case II with $\beta=\frac{1}{3}$.}
\label{fig: volume}
\end{figure}
\begin{figure}[htp!]
\centering
\includegraphics[width=0.5\textwidth]{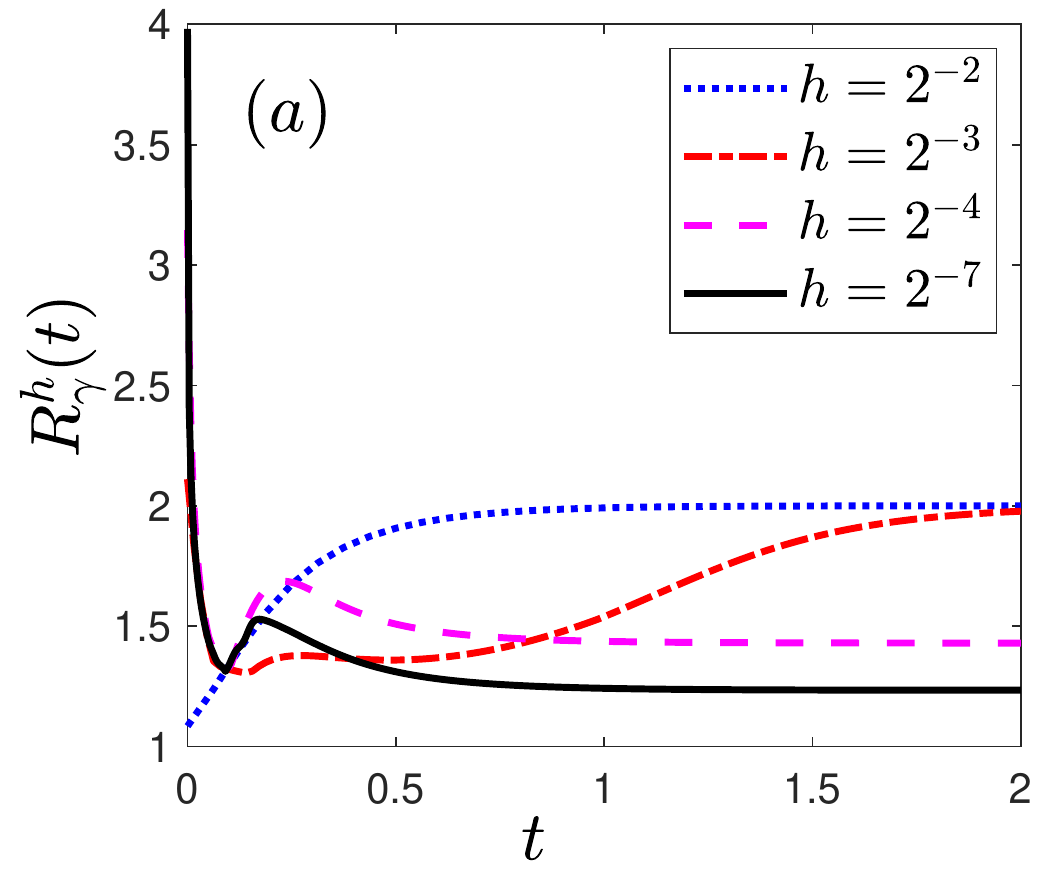}\includegraphics[width=0.5\textwidth]{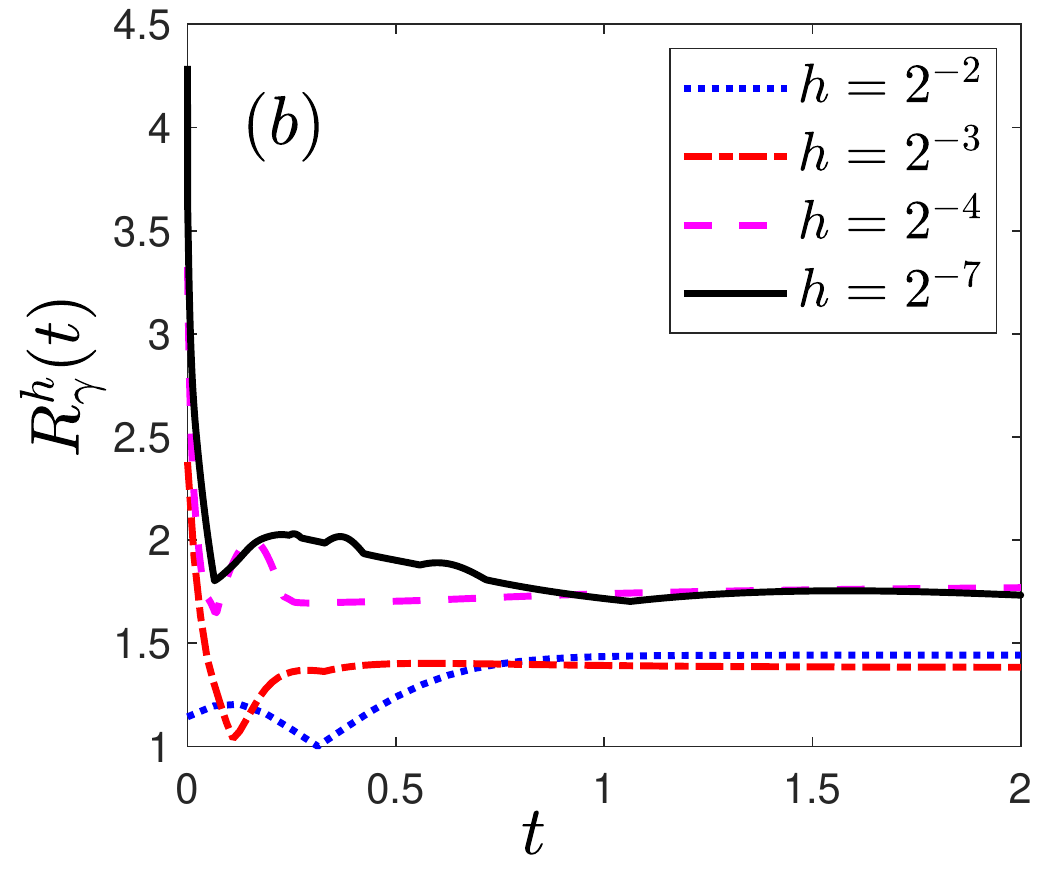}
\caption{Weighted mesh ratio of the SP-PFEM \eqref{eqn:2dfull sd} with $k(\boldsymbol{n})=k_0(\boldsymbol{n})$ and different $h$ for: (a) anisotropy in Case I; (b) anisotropy in Case II with $\beta=\frac{1}{3}$.}
\label{fig: mesh ratio}
\end{figure}
\subsection{Application for morphological evolutions}Finally, we apply the proposed SP-PFEMs to simulate the morphological evolutions driven by the three anisotropic geometric flows. Fig \ref{fig: evolve} plots the morphological evolutions of anisotropic surface diffusion for the four different anisotropic energies: (a) anisotropy in case I; (b)-(d) anisotropies in case II  with $\beta=1/9, 1/7, 1/3$, respectively. Fig \ref{fig: evolve full} and Fig \ref{fig: evolve full2} depict the anisotropic surface diffusion,
area-conserved anisotropic curvature flow, and anisotropic curvature flow at different times with anisotropy in case I and in case II with $\beta=1/3$, respectively.

\begin{figure}[htp!]
\centering
\includegraphics[width=1\textwidth]{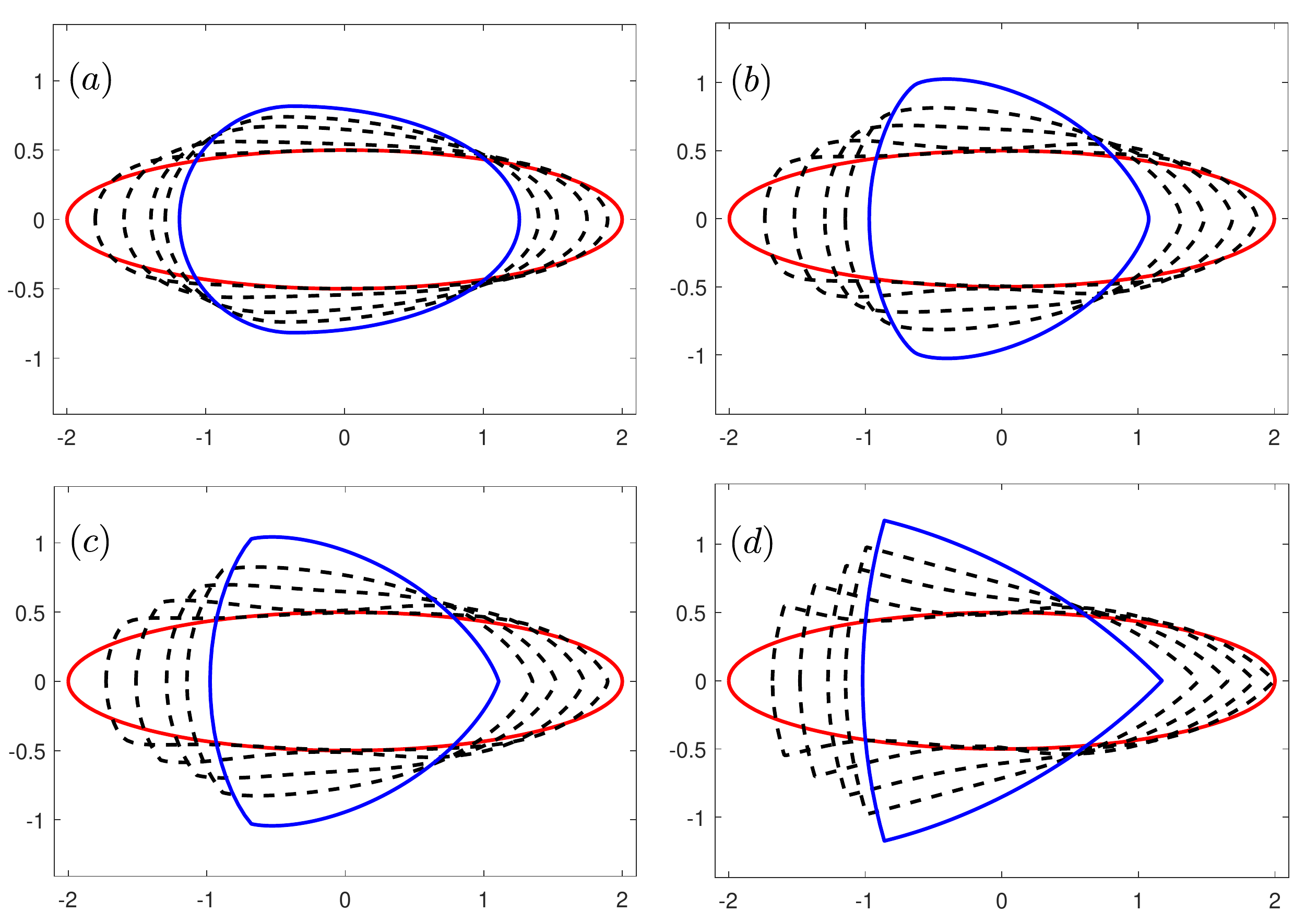}
\caption{Morphological evolutions of a $4\times 1$ ellipse under anisotropic surface diffusion with four different anisotropic energies: (a) anisotropy in Case I; (b)-(d) anisotropies in case II  with $\beta=1/9, 1/7, 1/3$, respectively. The red and blue lines represent the initial curve and the numerical equilibrium, respectively; and the black dashed lines represent the intermediate curves. The mesh size and time step are taken as $h=2^{-7}, \tau=h^2$.}
\label{fig: evolve}
\end{figure}

As shown in Fig. \ref{fig: evolve} (b)-(d), the edges emerge during the evolution and corners become sharper as the strength $\beta$ increases. In contrast, there are no edges or corners in the morphological evolutions with anisotropy in Case I. This suggests that even if it is not a $C^2$-function, it is more like weak anisotropy! From Fig. \ref{fig: evolve full}-\ref{fig: evolve full2}, we can see that the anisotropic surface diffusion and the area-conserved anisotropic curvature flow have the same equilibriums in shapes, while they have different dynamics, i.e., the equilibriums are different in positions, and the anisotropic surface diffusion evolves faster than the area-conserved anisotropic curvature flow.

\begin{figure}[htp!]
\centering
\includegraphics[width=0.33\textwidth]{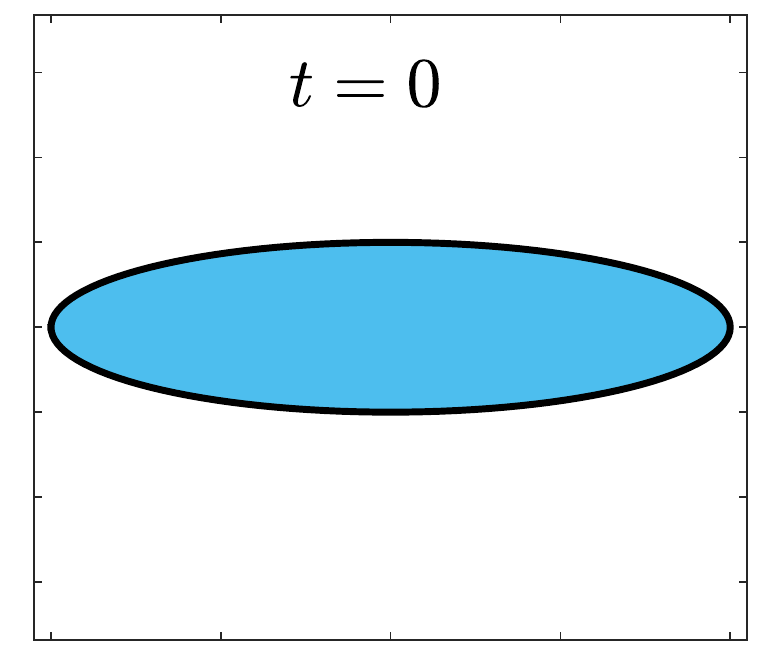}\includegraphics[width=0.33\textwidth]{figures/evolvenew/evolve1/a1.pdf}\includegraphics[width=0.33\textwidth]{figures/evolvenew/evolve1/a1.pdf}
\includegraphics[width=0.33\textwidth]{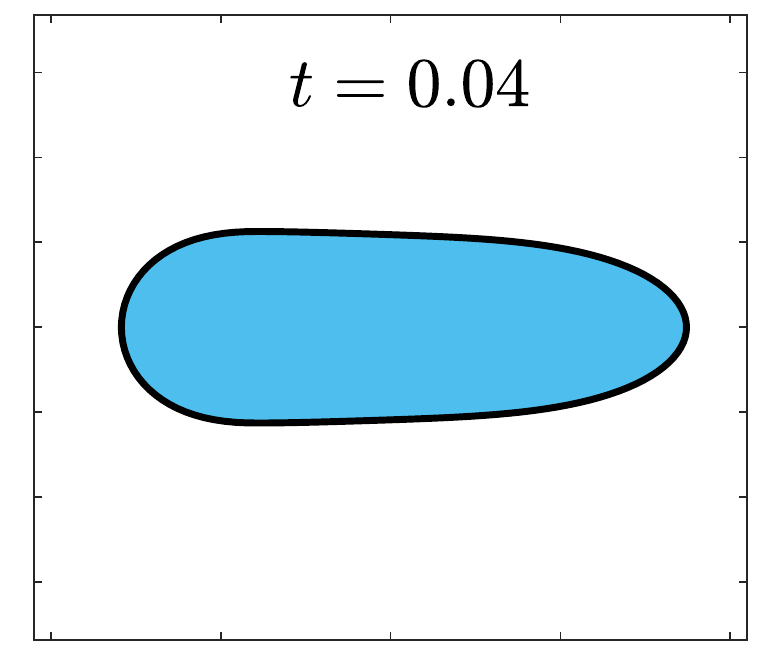}\includegraphics[width=0.33\textwidth]{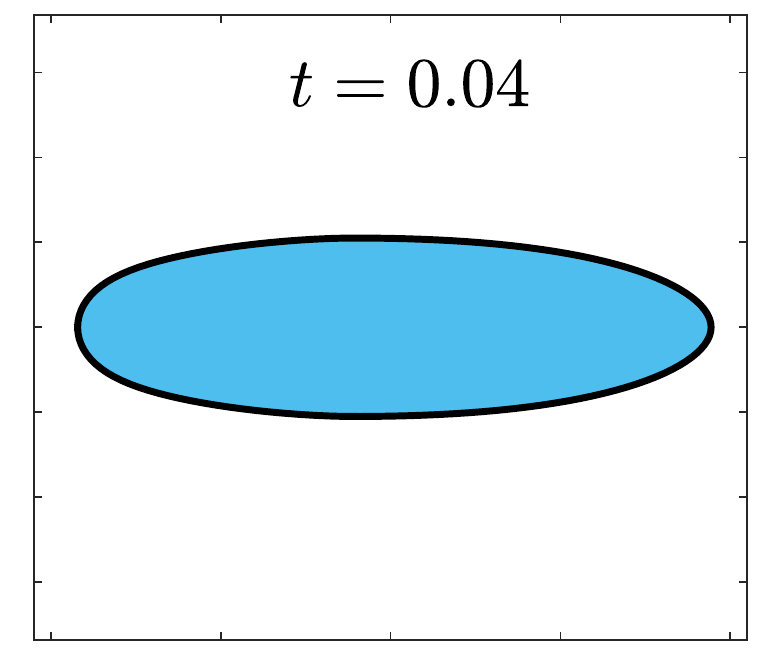}\includegraphics[width=0.33\textwidth]{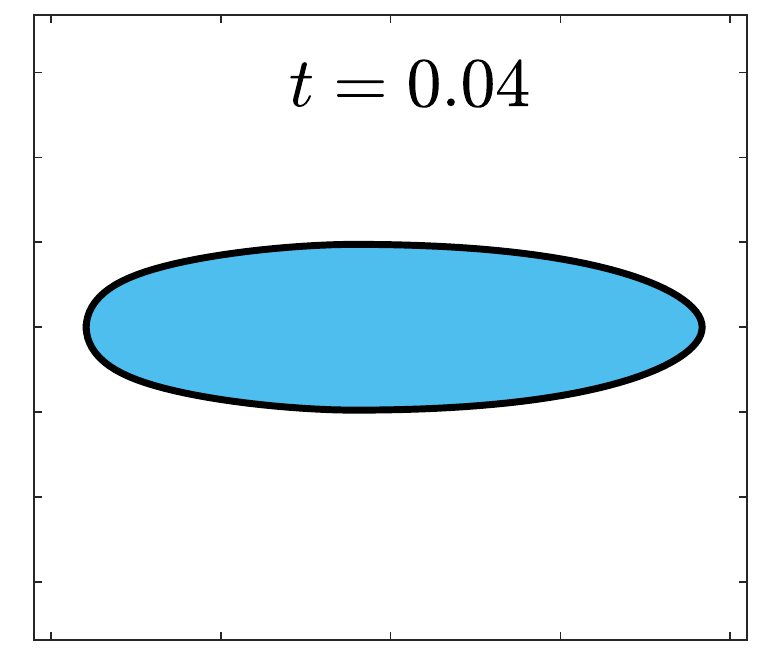}
\includegraphics[width=0.33\textwidth]{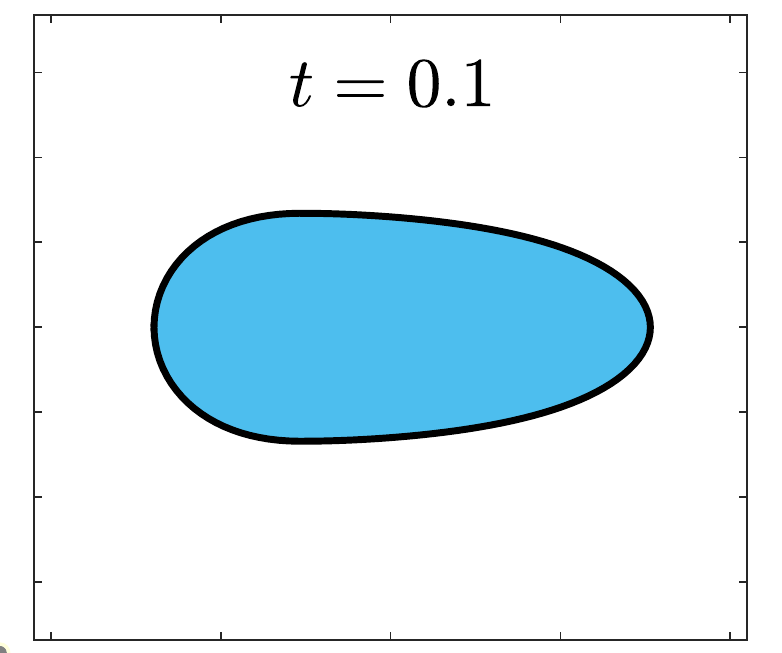}\includegraphics[width=0.33\textwidth]{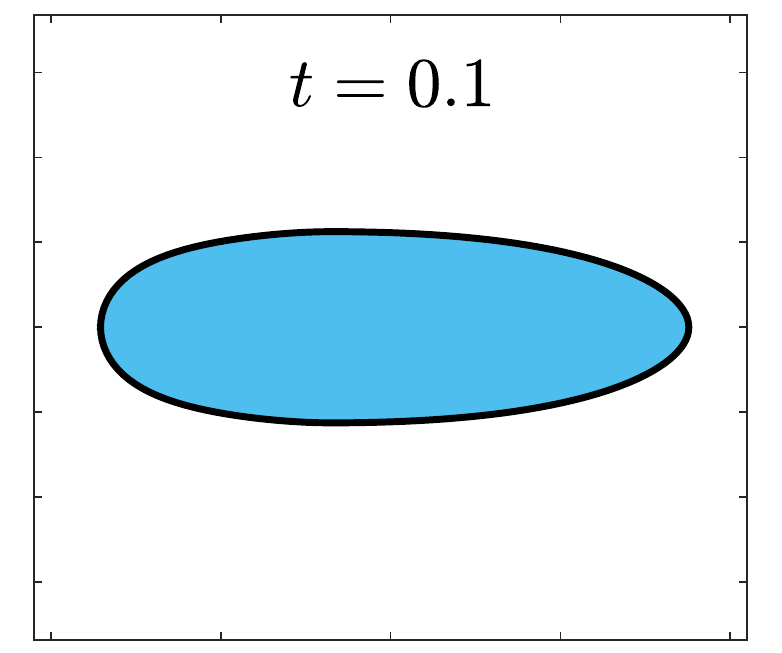}\includegraphics[width=0.33\textwidth]{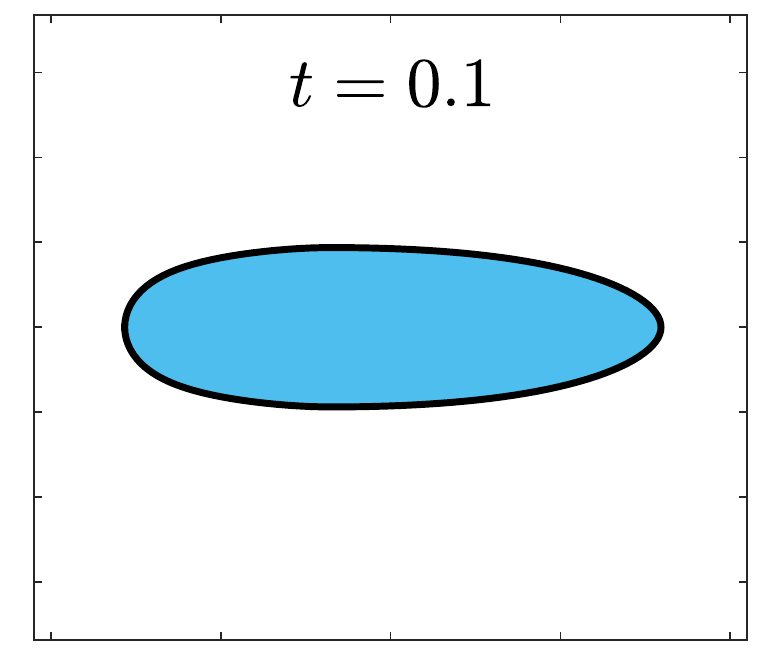}
\includegraphics[width=0.33\textwidth]{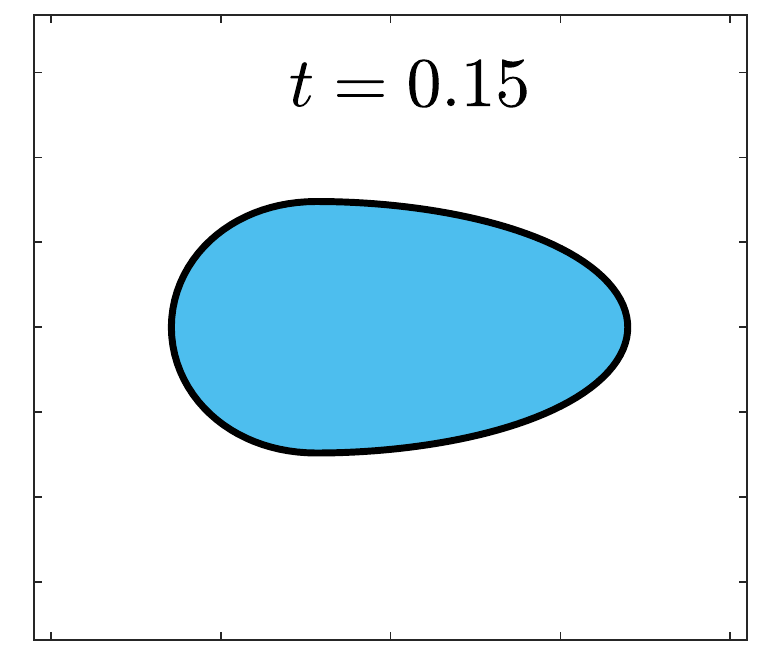}\includegraphics[width=0.33\textwidth]{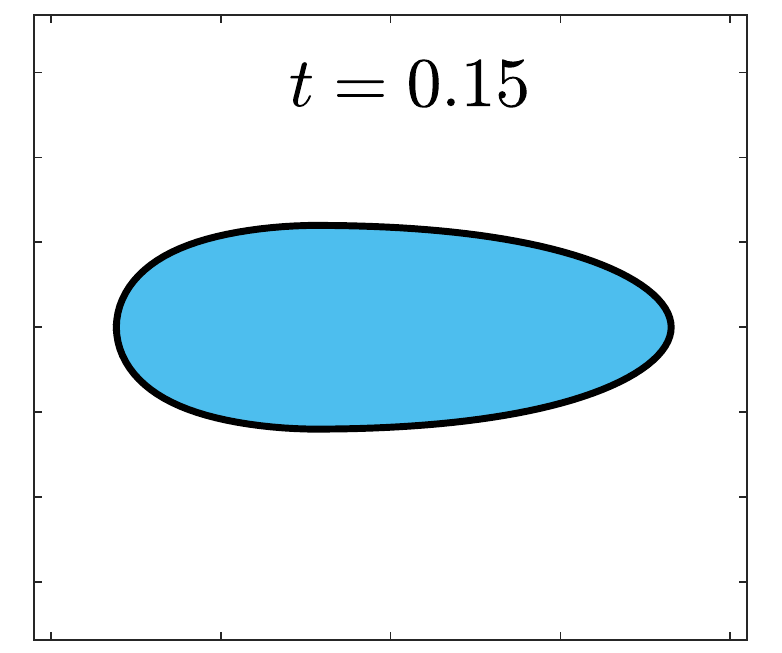}\includegraphics[width=0.33\textwidth]{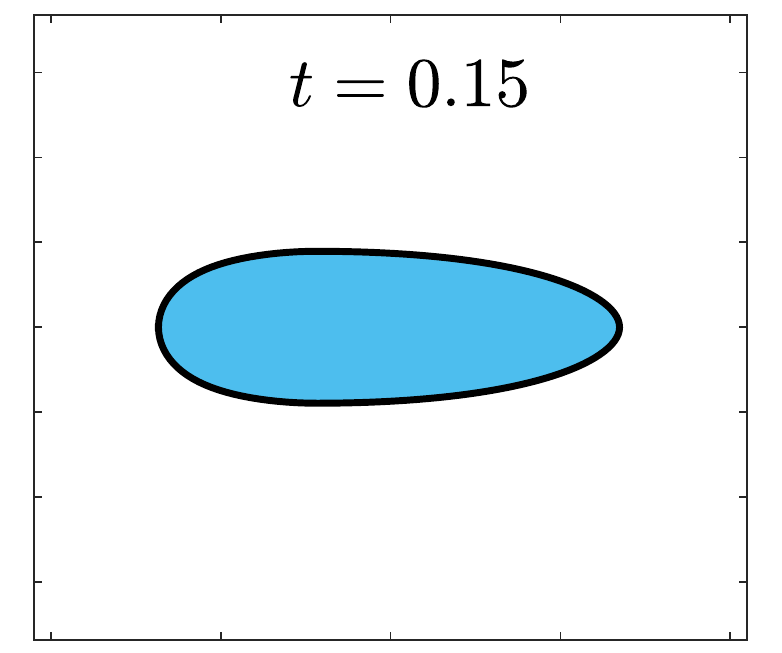}
\includegraphics[width=0.33\textwidth]{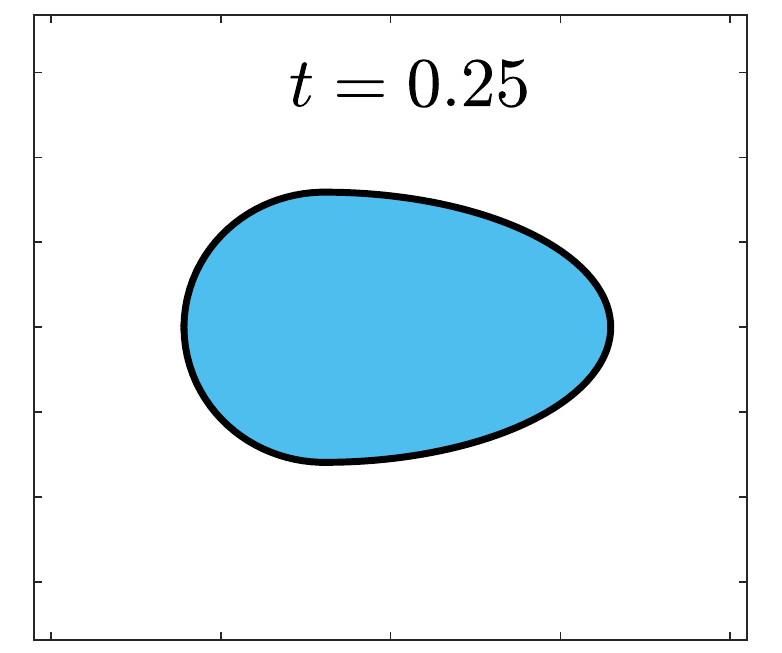}\includegraphics[width=0.33\textwidth]{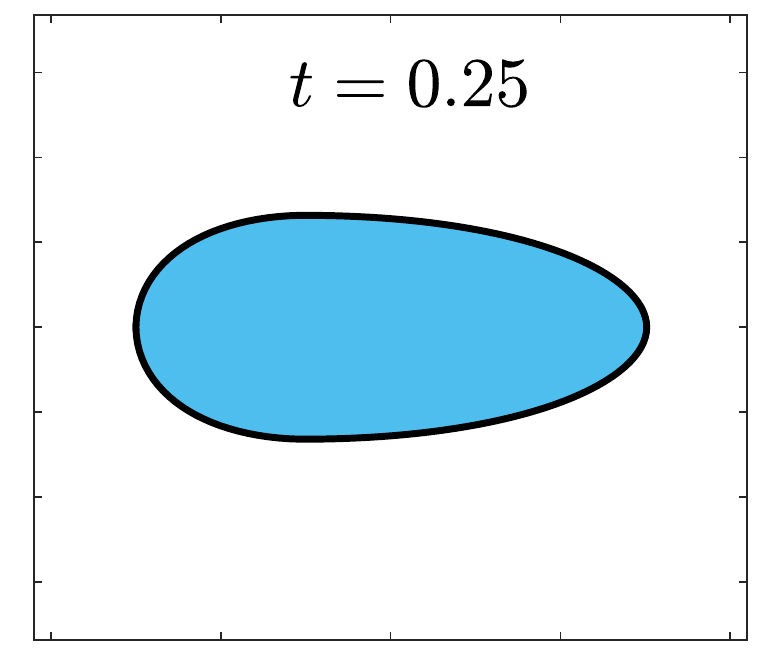}\includegraphics[width=0.33\textwidth]{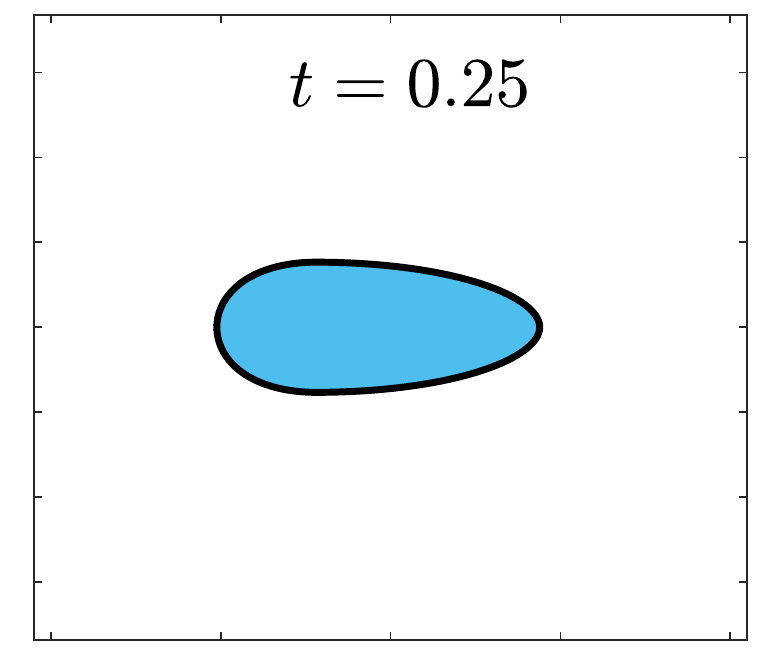}
\includegraphics[width=0.33\textwidth]{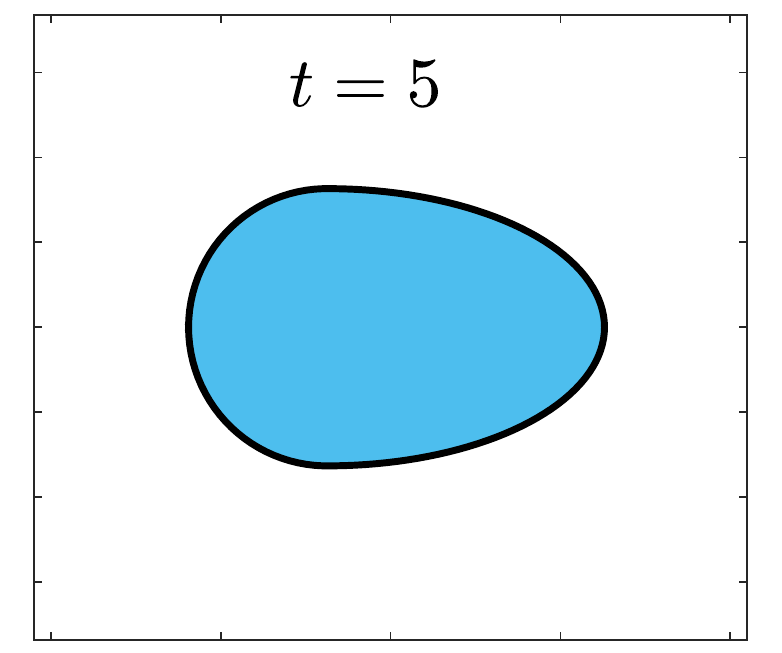}\includegraphics[width=0.33\textwidth]{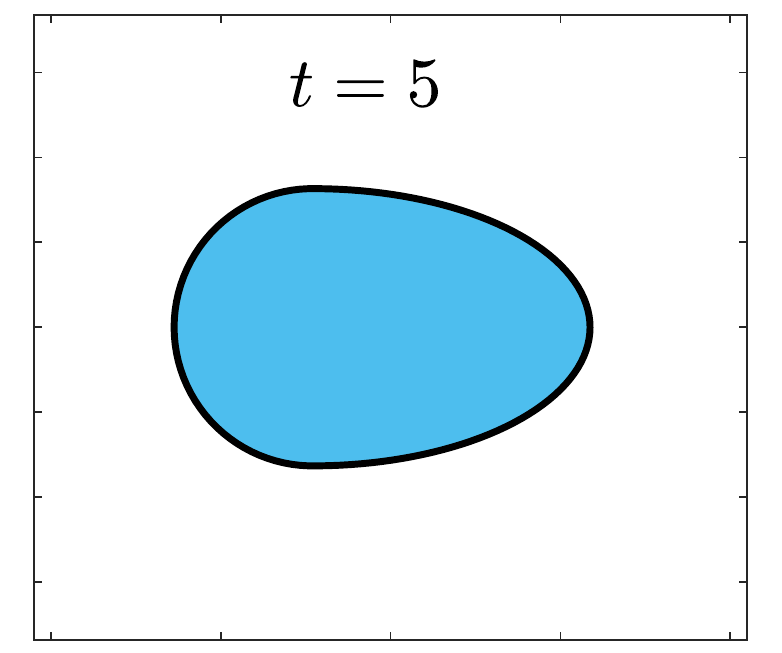}\includegraphics[width=0.33\textwidth]{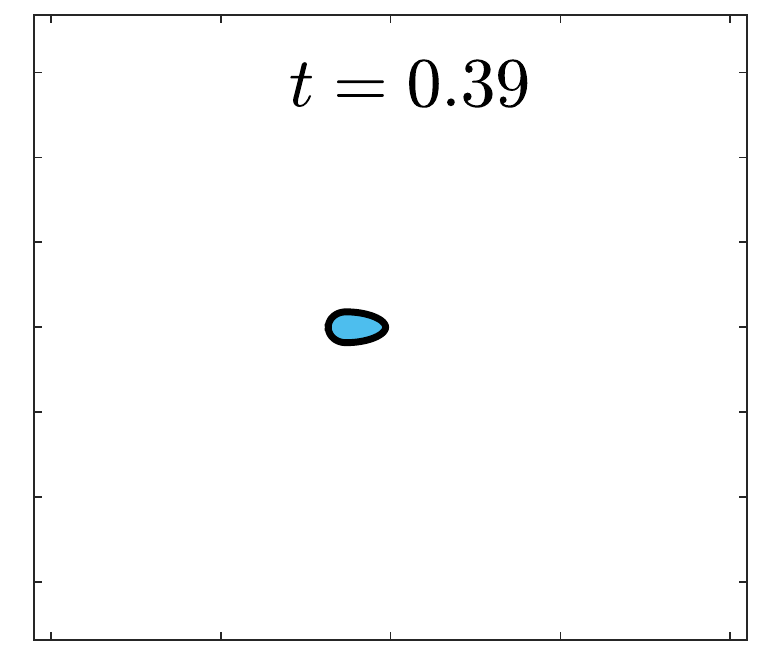}
\caption{Morphological evolutions of a $4\times 1$ ellipse under
anisotropic surface diffusion (left column), area-conserved anisotropic curvature flow (middle column) and anisotropic curvature flow (right column)
with the anisotropic surface energy in Case I at different times.
The evolving curves and their enclosed regions are colored by blue and black. The mesh size and time step are taken as $h=2^{-7}, \tau=0.001$.}
\label{fig: evolve full}
\end{figure}

\begin{figure}[htp!]
\centering
\includegraphics[width=0.33\textwidth]{figures/evolvenew/evolve1/a1.pdf}\includegraphics[width=0.33\textwidth]{figures/evolvenew/evolve1/a1.pdf}\includegraphics[width=0.33\textwidth]{figures/evolvenew/evolve1/a1.pdf}
\includegraphics[width=0.33\textwidth]{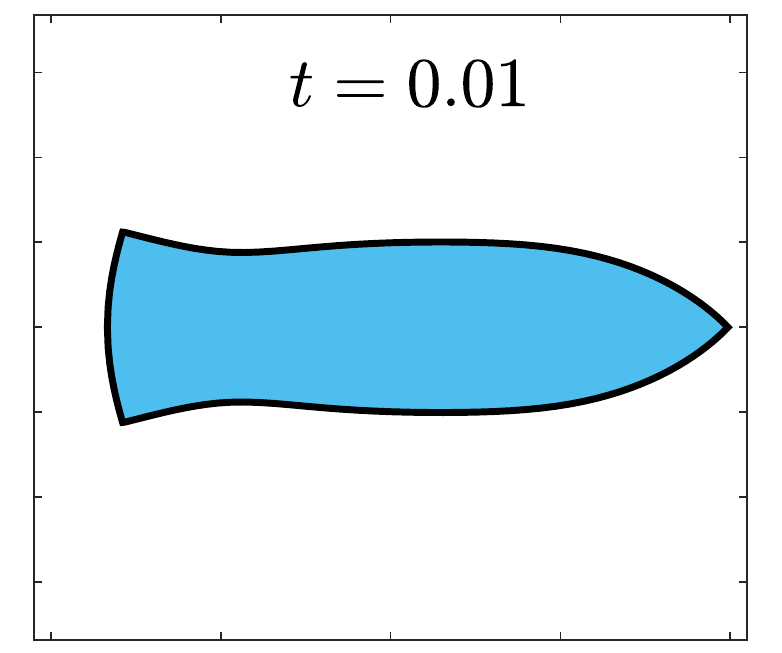}\includegraphics[width=0.33\textwidth]{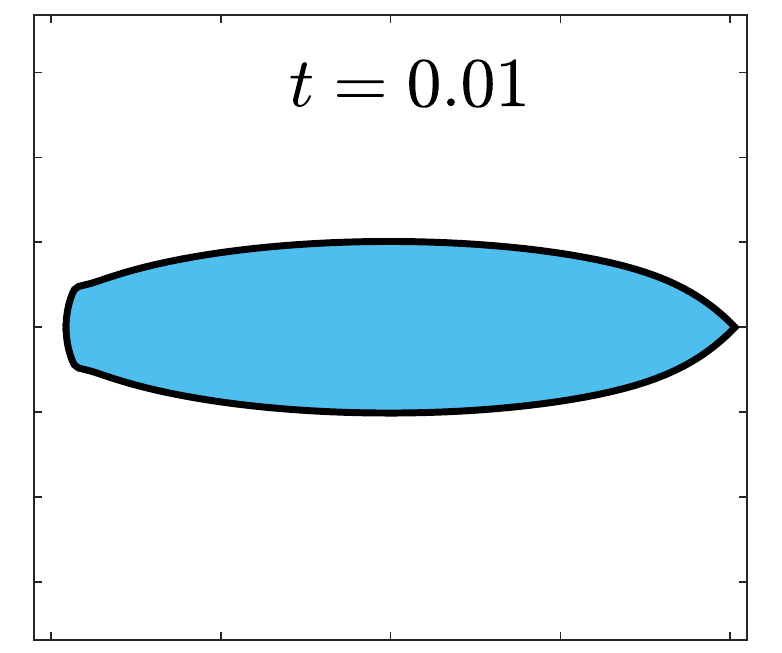}\includegraphics[width=0.33\textwidth]{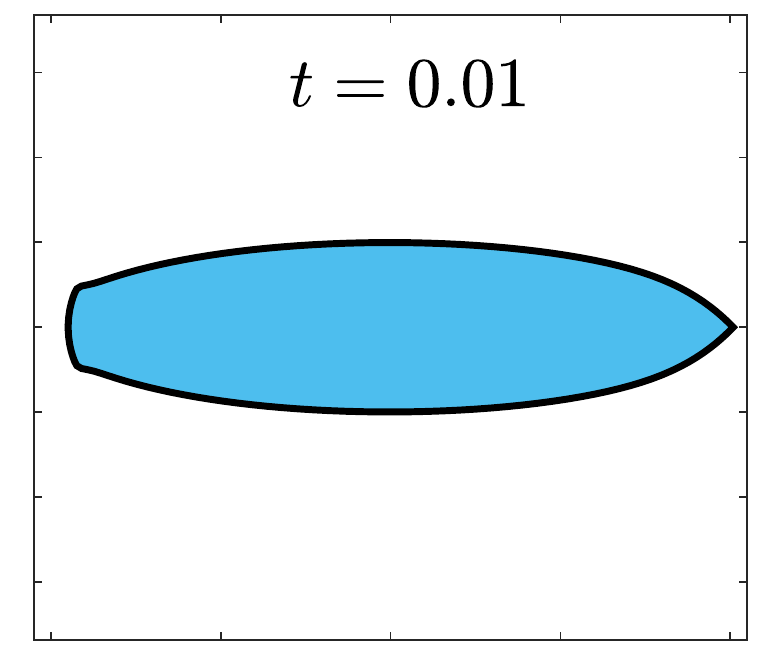}
\includegraphics[width=0.33\textwidth]{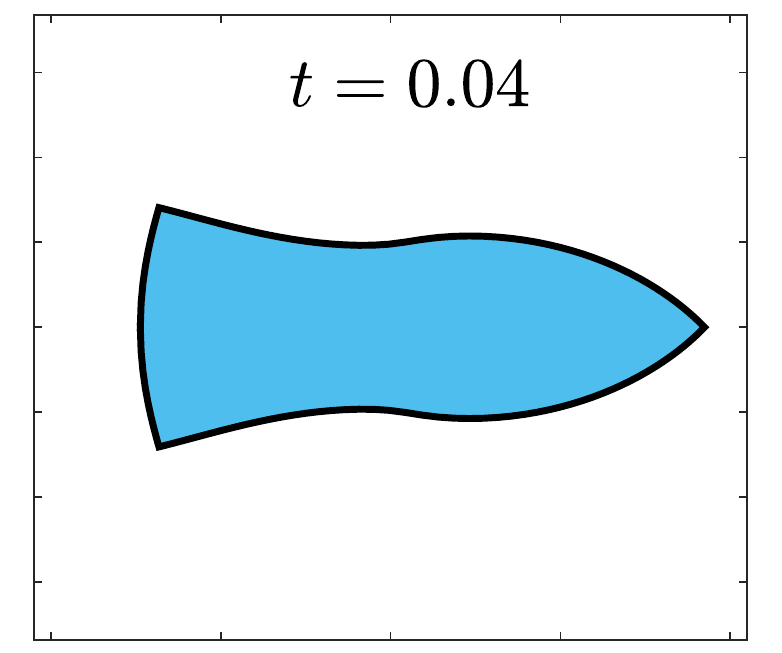}\includegraphics[width=0.33\textwidth]{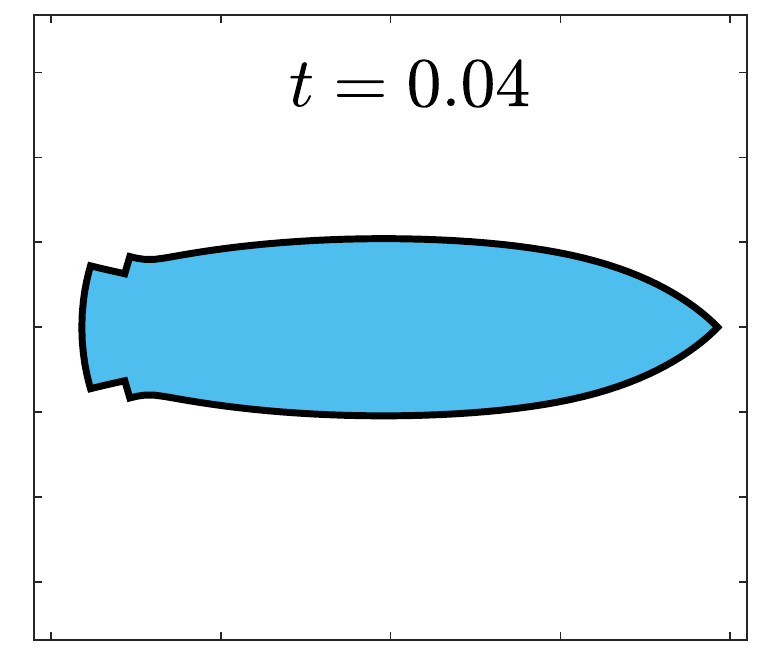}\includegraphics[width=0.33\textwidth]{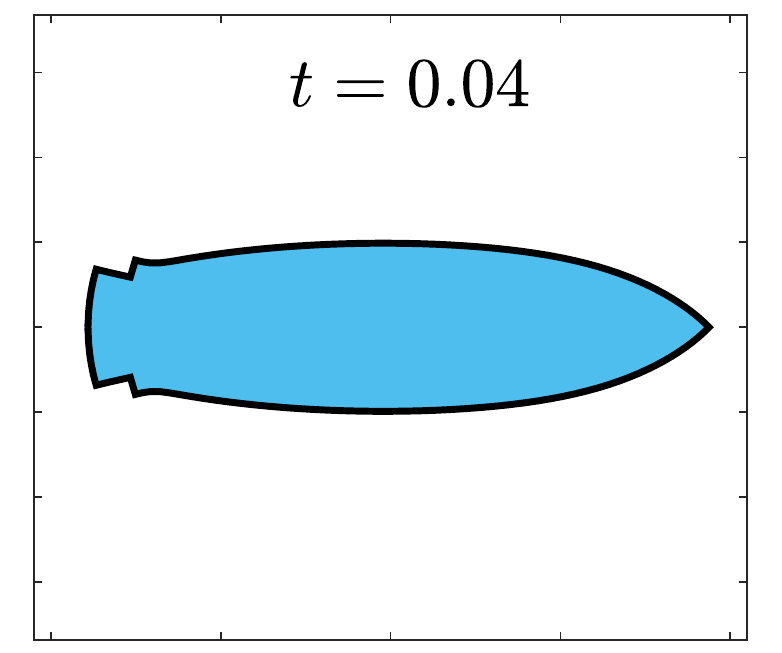}
\includegraphics[width=0.33\textwidth]{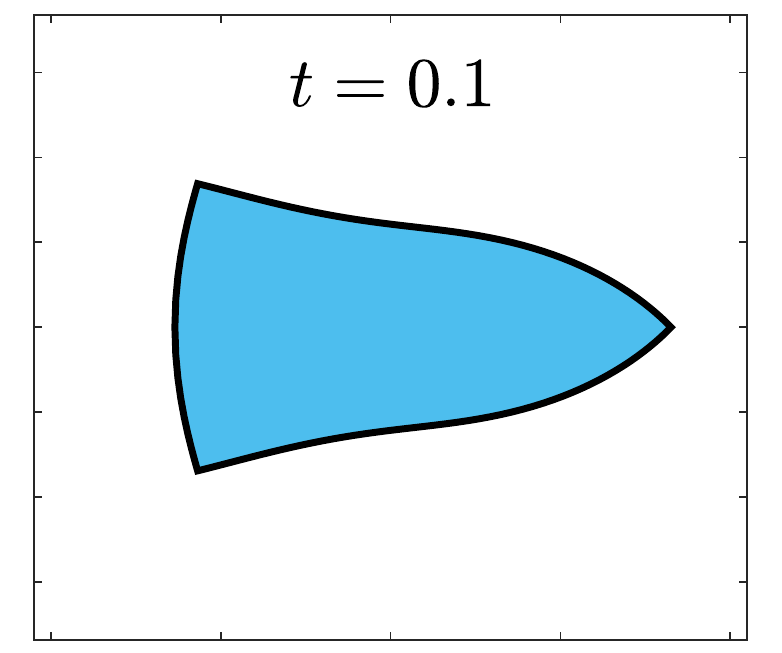}\includegraphics[width=0.33\textwidth]{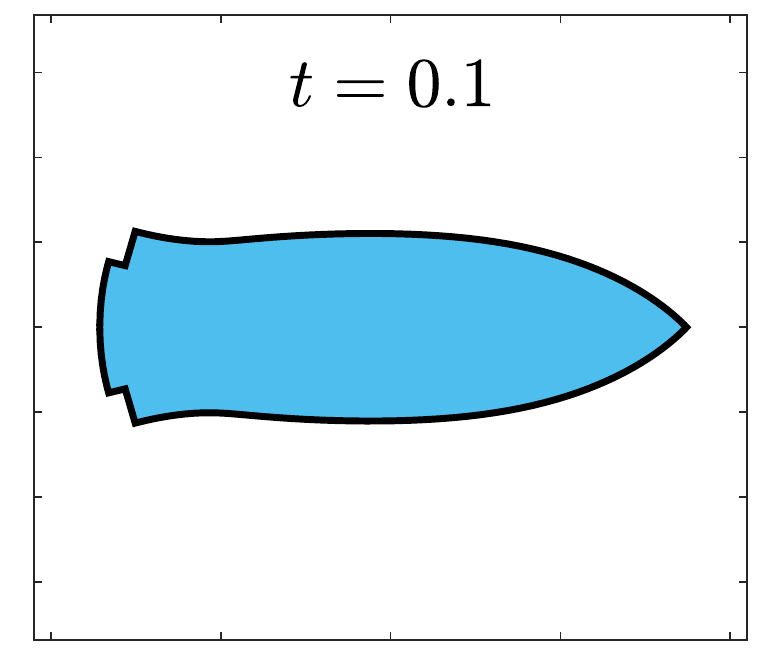}\includegraphics[width=0.33\textwidth]{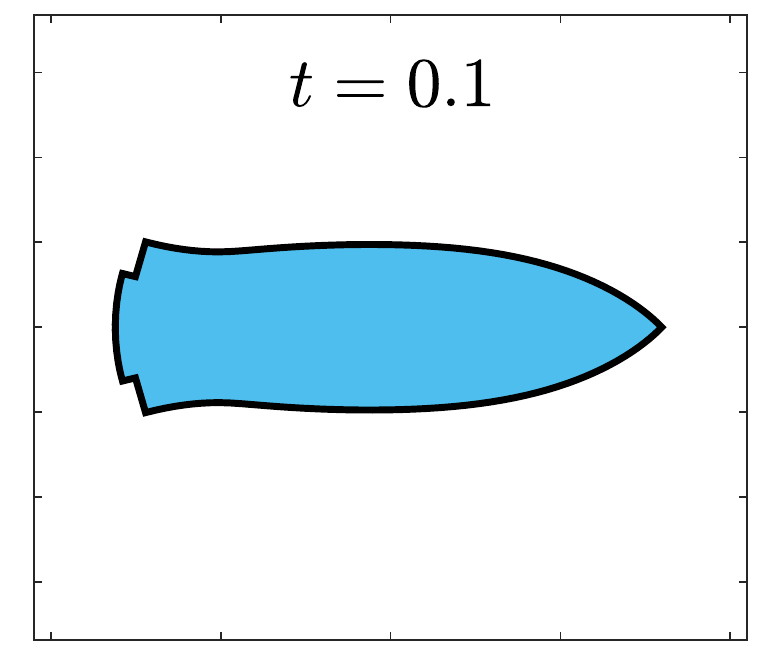}
\includegraphics[width=0.33\textwidth]{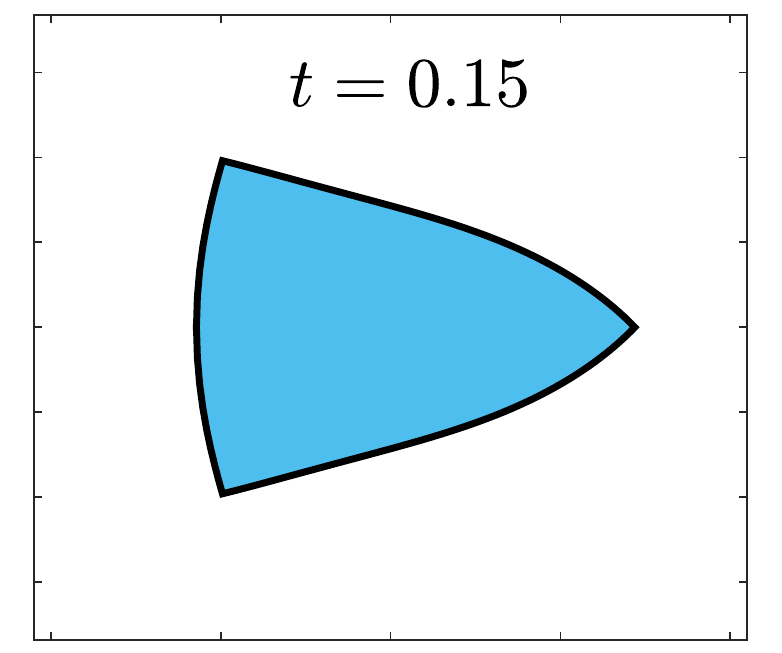}\includegraphics[width=0.33\textwidth]{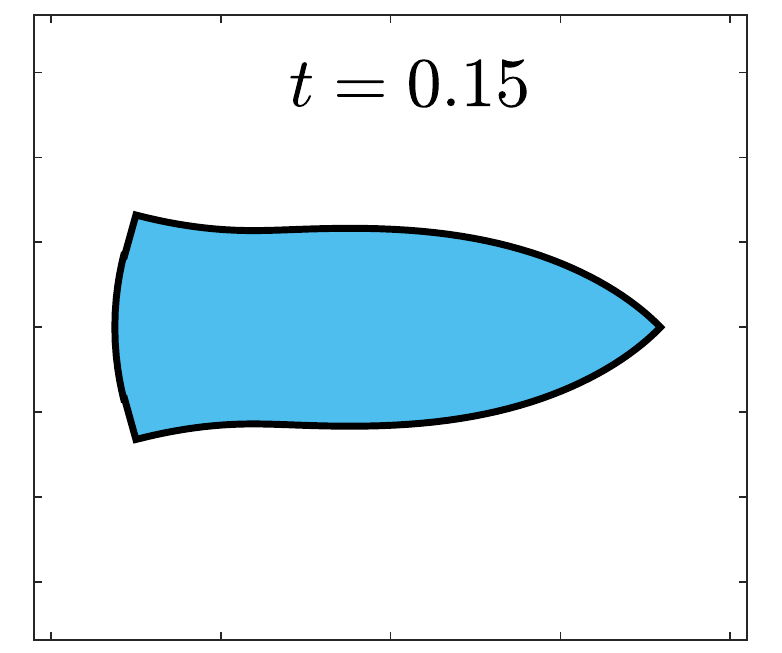}\includegraphics[width=0.33\textwidth]{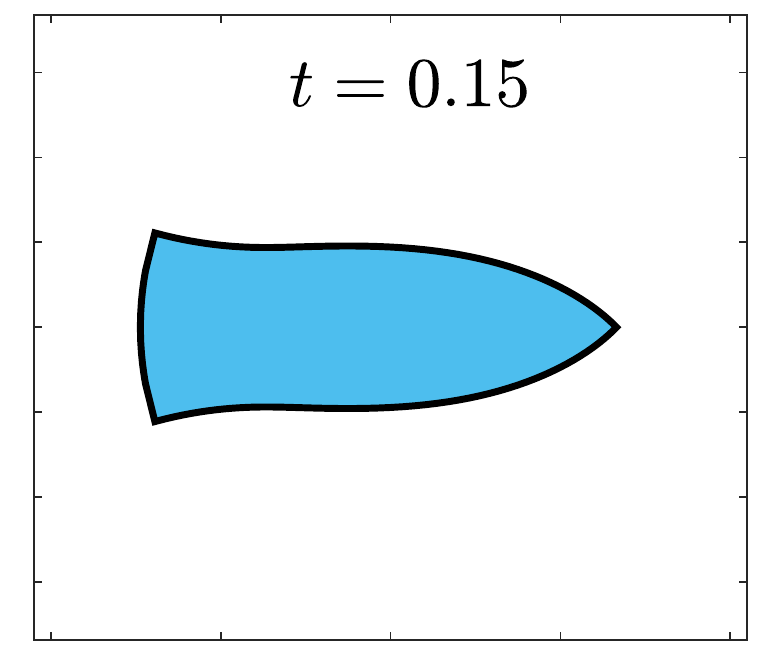}
\includegraphics[width=0.33\textwidth]{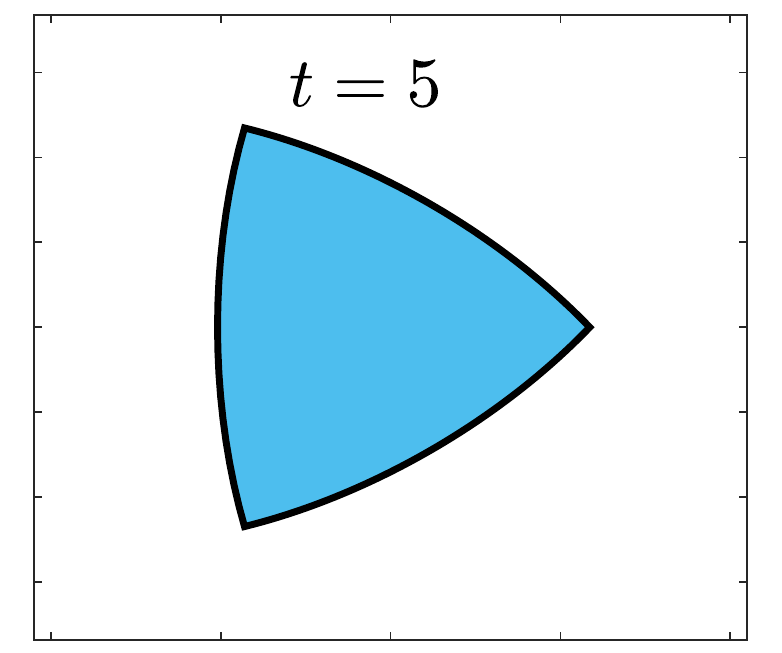}\includegraphics[width=0.33\textwidth]{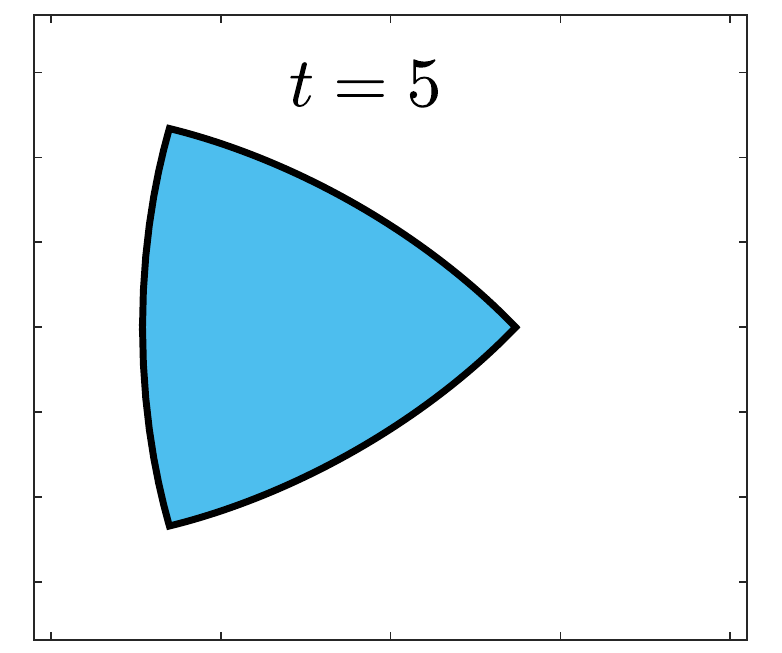}\includegraphics[width=0.33\textwidth]{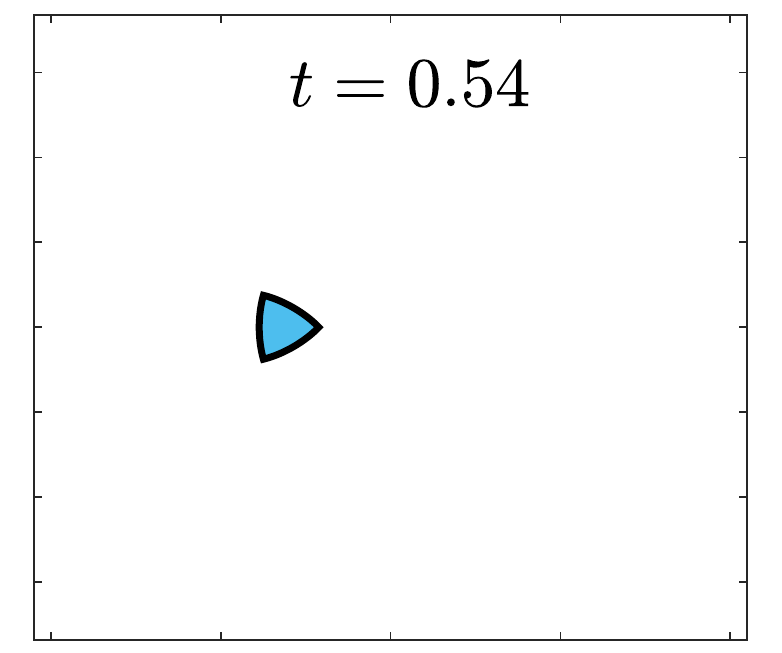}
\caption{Morphological evolutions of a $4\times 1$ ellipse under
anisotropic surface diffusion (left column), area-conserved anisotropic curvature flow (middle column) and anisotropic curvature flow (right column)
with the anisotropic surface energy in Case II with $\beta=1/3$ at different times.
The evolving curves and their enclosed regions are colored by blue and black. The mesh size and time step are taken as $h=2^{-7}, \tau=0.001$.}
\label{fig: evolve full2}
\end{figure}

\section{Conclusions}
By introducing a novel surface energy matrix $\boldsymbol{G}_k(\boldsymbol{n})$ depending
on the anisotropic surface energy $\gamma(\boldsymbol{n})$ and the Cahn-Hoffman $\boldsymbol{\xi}$-vector  as well as a nonnegative stabilizing function $k(\boldsymbol{n})$, we proposed conservative geometric partial differential equations for several geometric flows with anisotropic surface energy $\gamma(\boldsymbol{n})$. We derived their weak formulations and applied PFEM to get their full discretizations. Then we proved these PFEMs are structure-preserving under a very mild condition on $\gamma(\boldsymbol{n})$ with proper choice of the stabilizing function $k(\boldsymbol{n})$. Though our surface energy matrix $\boldsymbol{G}_k(\boldsymbol{n})$ is no longer symmetric, our experiments had shown a robust second-order convergence rate in space and linear convergence rate in time and unconditional energy stability. Specifically,
the mesh quality of the proposed SP-PFEM, i.e. the weighted mesh ratio is
much smaller, is much better than that in the symmetrized SP-PFEM proposed
recently for anisotropic surface diffusion with a symmetric surface energy
\cite{bao2021symmetrized,bao2022symmetrized}. Moreover, our SP-PFEMs work well for the piecewise $C^2$ anisotropy, which is a significant achievement compared with other PFEMs. In the future, we will generalize the surface energy matrix $\boldsymbol{G}_k(\boldsymbol{n})$ to three dimensions (3D) and
propose efficient and accurate SP-PFEM for anisotropic geometric flows in 3D.


\end{document}